\documentclass[11pt]{article}
\usepackage{amsmath,amsthm,amssymb,amsfonts,amsbsy,mathtools,enumerate}
\usepackage{tikz}
\usepackage{float}
\usepackage{graphicx}
\usepackage[mathlines]{lineno}
\usepackage{blkarray}
\usepackage{caption, subcaption}
\usepackage{nicematrix}
\usepackage[toc,page]{appendix}
\usepackage[hidelinks]{hyperref}
\usepackage{url}
\usepackage{tikz-cd}
\newtheorem*{theorem*}{Theorem}
\theoremstyle{plain}
\theoremstyle{definition}
\newtheorem{theorem}{Theorem}[section]
\newtheorem{corollary}[theorem]{Corollary}

\newtheorem{lemma}[theorem]{Lemma}
\newtheorem{prop}[theorem]{Proposition}

\newtheorem{defn}[theorem]{Definition}
\newtheorem{rk}{Remark}

\newcommand\sqr[2]{{\vbox{\hrule height.#2pt
    \hbox{\vrule width.#2pt height#1pt \kern#1pt
        \vrule width.#2pt}\hrule height.#2pt}}}
\newcommand\qedbox{%
	\ifmmode\eqno\sqr53
	\else\nolinebreak\ \hfill\sqr53\medbreak\fi}

\newcommand{\al}{\ensuremath{\alpha}}

\def\FF{{\mathbb F}}
\def\Q{{\mathbb Q}}
\def\Z{{\mathbb Z}}

\def\T{{\mathcal{T}}}
\newcommand\re{{\mathbb R}}%reals

\newcommand\hlcl[2]{H_1(\text{Cl}(#1), #2)} 
 
\newcommand\clc[1]{\text{Cl}(#1)}

\newcommand\rnkg[2]{\text{rank}_{#1}(#2)}
\newcommand\rankg[2]{\text{rank}_{\mathbb{#1}}(#2)}
\newcommand\fchar[1]{\text{char}(#1)}

\newcommand\oa[2]{\text{OA}({#1}, {#2})}
\DeclareMathOperator\im{im}

\newcommand\kdt[2]{\ker(\delta_{#1}^T(#2))}
\newcommand\dkdt[2]{\dim(\ker(\delta_{#1}^T(#2)))}
\newcommand\idt[2]{\im(\delta_{#1}^T(#2))}
\newcommand\didt[2]{\dim(\im(\delta_{#1}^T(#2)))}

\newcommand\tr[1]{\text{tr}(#1)}

\title{On the homology groups of clique complexes of strongly regular graphs}
\author{Sebastian M. Cioab\u{a} and Mutasim Mim}
\date{\today}

\begin{document}
\maketitle

\begin{abstract}
For a graph $G$ and a field $\mathbb{F},$ the first clique-homology $H_1(\text{Cl}(G),\mathbb{F})$ vanishes precisely when the cycle space of $G$ over $\mathbb{F}$ is generated by the signed boundaries of the triangles. We develop cycle-surgery methods for establishing this property in arbitrary characteristic and apply them to strongly regular graphs. Combining our results with Neumaier's classification, we show that $H_1(\text{Cl}(G),\mathbb{F}) \ne 0$ can only occur in the Petersen graph, the Shrikhande graph, the complete bipartite graphs, the conference graphs on at most $255$ vertices, the lattice graphs, and the finite exceptional families $E_m$ in Neumaier's classification of strongly regular graphs with smallest adjacency eigenvalue $-m$, for some integer $m \geq 3$. Consequently, if $(G_n)_{n\geq 1}$ is an infinite family of pairwise distinct strongly regular graphs and $(\mathbb{F}_n)_{n\geq 1}$ is a sequence of fields such that $H_1(\text{Cl}(G_n), \mathbb{F}_n)\not=0$ for every $n$, then either $G_n$ is a lattice graph for infinitely many $n$, or $\lim_{n\rightarrow +\infty} \lambda_{\min}(G_n)=-\infty$. 

For Latin square graphs, we determine the clique homologies over arbitrary fields and show that if $G$ is the strongly regular graph associated with a Latin square $M$ of order $n \geq 5$ and $\mathbb{F}$ is any field, then $H_i(\text{Cl}(G),\mathbb{F})=0$ for $i=1$ or $i \geq 3$, and $\dim  H_2(\text{Cl}(G),\mathbb{F})=(n-1)^3-I(M),$ where $I(M)$ is the number of $2 \times 2$ Latin subsquares or intercalates in $M$.
\end{abstract}

\section{Introduction}

Let $G$ be a graph and $\clc{G}$ be its clique complex. If $R$ is a ring, the vanishing of the first clique homology $\hlcl{G}{R}$ over $R$ phenomenon has been studied in several contexts, particularly for $\mathbb{Z}$, $\FF_2,$ the field with two elements, and $\re$, see \cite{Duchet,Kahle,Meshulam}. The assertion $\hlcl{G}{R}=0$ is equivalent to the cycle space of the graph $G$ being spanned by the signed boundaries of the triangles in $G$.
%(the image of the boundary matrix $\delta_{1,G,\FF}^T$). 
% the boundary matrix has not been defined yet

While they did not explicitly refer to clique complexes, homology, or cohomology, Duchet, Las Vergnas, and Meyniel \cite{Duchet} obtained a characterization of planar graphs $G$ with $\hlcl{G}{\FF_2}=0$ and called such graphs nullhomotopic. Meshulam \cite{Meshulam} proved that any graph $G$ with the property that every four vertices have a common neighbor satisfies $\hlcl{G}{\FF}=0$ for every field $\FF$. Linial and Meshulam \cite{LinialMeshulam} showed that in a random $2$-complex $Y$, $H_1(Y,\FF_2)=0$ has a sharp threshold $p \sim \frac{2\log(n)}{n}$. In \cite{cioaba-guo-ji-mim}, we strengthened Meshulam's result over $\re$ and proved that any graph $G$ with the property that every induced cycle has four consecutive vertices with a common neighbor satisfies $\hlcl{G}{\FF}=0$. We also showed that if $G$ is a strongly regular graph with parameters $(v,k,\lambda,\mu),$ and $2 \lambda + \mu + 2 > 2k,$ then  $\hlcl{G}{\re}=0$ and observed that the Paley graph $P(q)$ on $q$ vertices satisfies $\hlcl{P(q)}{\re}=0$ for $q \geq 4097$.

In this paper, we study the vanishing $\hlcl{G}{\FF}=0$ phenomenon for strongly regular graphs $G$ and an arbitrary field $\FF$. We combine our vanishing $H_1$ results with Neumaier's classification of strongly regular graphs to isolate the strongly regular graphs $G$ for which $\hlcl{G}{\FF}$ may be nonzero for some $\FF$. In particular, we prove that lattice graphs and complete bipartite graphs have non-vanishing $H_1$. Furthermore, outside of these two families and a finite exceptional family, any strongly regular $G$ with $\hlcl{G}{\FF}\ne 0$ must have integral adjacency eigenvalues, and for every fixed $m \geq 3,$ there are only a finite number of such graphs $G$ with smallest eigenvalue $-m$. 

% Our proofs are combinatorial and are based on local surgery of cycle lemmas. In particular, our techniques do not require the inner product structure of the ambient field $\FF,$ and are valid over every field. The condition $\hlcl{G}{\FF}=0$ over every field in particular implies that $\hlcl{G}{\re}=0$. As a consequence of the vanishing $H_1$ result, we obtain a complete description of the up-Laplacian spectrum at every level, and all homologies of Latin square graphs of order $n \geq 5.$ We show that even though the first homology vanishes, the second homology $H_2$ counts the number of $2 \times 2$ Latin subsquares. Three infinite families of strongly regular graphs appearing in Neumaier's classification are graphs associated with orthogonal arrays, Steiner systems, and conference graphs. We prove our vanishing results for these families separately. 

Our proofs are combinatorial and are based on local surgery of cycle lemmas. In particular, our techniques do not require an inner product structure of the ambient field $\FF,$ and are valid over every field.  Our starting point is a collection of cycle-surgery lemmas. They replace cycles by shorter cycles after adding signed triangle boundaries. The arguments are algebraic over the coefficient ring: they do not use an inner product and remain valid over arbitrary characteristic. One consequence is the following vanishing criterion: if every induced cycle of a graph has four consecutive vertices with a common neighbor, then its cycle space is triangle-boundary generated over every field. For strongly regular graphs, it is enough to verify an analogous local condition only for induced cycles of lengths four and five (see Theorems \ref{thm:4-vertex-common-neighbor} and \ref{thm:4-vertex-common-neighbor-srg}). We apply this framework to the principal infinite families in Neumaier's classification theorem for strongly regular graphs. This yields our main classification theorem (Theorem \ref{thm:classification-theorem-1}) and its asymptotic counterpart (Theorem \ref{thm:classification-theorem-2}):

\begin{theorem*}[Classification theorem]
    Let $G$ be a strongly regular graph and $\FF$ a field. If $\hlcl{G}{\FF} \not = 0$, then $G$ must be one of the following graphs:
    \begin{enumerate}
       \item the Petersen graph, with parameters $(v,k,\lambda,\mu)=(10,3,0,1),$
       \item the Shrikhande graph, with parameters $(v,k,\lambda,\mu)=(16,6,2,2),$
       \item an element of the exceptional family $E_m$, for some $m \geq 3$, described in the Neumaier's classification,
       \item a conference graph $G(v, \frac{v-1}{2}, \frac{v-5}{4}, \frac{v-1}{4})$ for some $v \leq 255$,
       \item a complete bipartite graph with each part of size $n$ and parameters $(2n, n, 0, n),$
       \item a lattice graph $L_2(n)=K_n \square K_n,$ with parameters $(v,k,\lambda,\mu)=(n^2,2(n-1),n-2,2), \ n \geq 3$.
    \end{enumerate}
\end{theorem*}

\begin{theorem*}[Infinite-family dichotomy]
    If $(G_n)_{n=1}^{\infty}$ is a sequence of pairwise distinct strongly regular graphs and $(\FF_n)_{n=1}^{\infty}$ is a sequence of arbitrary fields such that $\hlcl{G_n}{\FF_n}\not =0$ for every $n$, then one of the following statements is true:
    \begin{enumerate}
        \item for infinitely many values of $n$, $G_n$ is a lattice graph with at least nine vertices,  
        \item $\lim_{n \to \infty}\lambda_{\min}(G_n)=-\infty.$
    \end{enumerate}
\end{theorem*}

We prove the vanishing first homology results separately for the infinite families appearing in Neumaier's classification. We show that graphs arising from orthogonal arrays and block graphs of Steiner systems have triangle-boundary generated cycles space. We prove that the same is true for conference graphs on at least $256$ vertices. For strongly regular graphs with smallest eigenvalue $-2,$ Seidel's classification leaves only the Petersen graph, the Shrikhande graph, and the lattice graphs as possible graphs with a nontrivial $H_1$ group.

Our second main result (Theorem \ref{theorem:homology-latin-square}) gives considerably more information for Latin square graphs. If $M$ is a Latin square of order $n$, its Latin square graph has $n^2$ coordinate positions of $M$ as vertices, with two cells adjacent when they lie in the same row, lie in the same column, or contain the same symbol. An intercalate in $M$ is a $2 \times 2$ Latin subsquare of $M$, and we write $I(M)$ for the number of intercalates in $M$.

\begin{theorem*}[Homology of Latin square graphs]
    Let $M$ be a Latin square of order $n \geq 5$ and $G$ the associated strongly regular graph. Let $\FF$ be an arbitrary field. Denote by $I(M)$ the number of intercalates in $M$. Then, for $i > 0$,
    \begin{gather*}
        \dim(H_i(\clc{G}, \FF)) = \begin{cases}
            (n-1)^3 - I(M), & \text{if } i = 2\\
            0, & i=1 \text{ or } i \geq 3.
        \end{cases}
    \end{gather*}
\end{theorem*}

The present work grew out of \cite{cioaba-guo-ji-mim}, but its principal results and methods are different. The previous work emphasized spectra of higher Laplacians and real cohomology for selected families of strongly regular graphs. Here the new ingredients are the characteristic-free cycle-surgery method, the vanishing theorems for the families occurring in Neumaier's classification, the resulting global classification, the infinite-family dichotomy, and the complete homology calculation for Latin square graphs. In particular, the classification theorems above are not consquences of the spectral computations of \cite{cioaba-guo-ji-mim}.
\color{black}
The paper is organized as follows. 

In Section \ref{sec:notations-and-background}, we introduce the notations and terminology used throughout the paper and present several results related to the cycle space of a graph. In subsection \ref{subsec:a-vanishing-theorem}, we obtain a vanishing $H_1$ criteria for strongly regular graphs.

In Section \ref{sec:latin-square-graphs}, we prove our results regarding Latin square graphs. In particular, in subsection \ref{subsec:the-h1-of-clique-complex-of-latin-square-grpahs} we prove the vanishing $H_1$ result for Latin square graphs and in subsection \ref{theorem:h-2-latin-square-bounds} we show that $H_2$ counts the number of Latin subsquares or intercalates. In subsection \ref{subsec:higher-homologies-of-Latin-square-graphs}, we prove that all higher homologies vanish. 

In Sections \ref{sec:graphs-associated-with-orthogonal-arrays}, \ref{sec:block-graphs-of-steiner-systems}, and \ref{sec:conference-graphs}, we prove the vanishing results for graphs associated with orthogonal arrays, Steiner system block graphs, and conference graphs, respectively. 

In Section \ref{sec:strongly-regular-graphs-with-smallest-positive-eigenvalue--2}, we combine our vanishing results with Seidel's classification of strongly regular graphs with smallest adjacency eigenvalue $-2,$ and isolate graphs in this family that have non-vanishing $H_1$ to six explicit graphs and the lattice graphs. 

In Section \ref{sec:a-classification-theorem}, we combine our results with Neumaier's classification, and isolate strongly regular graphs with non-vanishing $H_1$ to explicit families and the exceptional families appearing in Neumaier's classification.

\section{Notations and Background} \label{sec:notations-and-background}

\subsection{Graphs and Simplicial Complexes} 

In this paper, graphs refer to finite, simple graphs, i.e., graphs without loops and multiple edges. Given a ground set $V,$ a simplicial complex $X$ on $V$ is a family of subsets of $V$ such that: if $F \in X$ and $E \subset F,$ then $E \in X.$ The elements in $X$ are called simplices, and the dimension of a simplex $F \in X$ is defined to be $|F|-1.$ The set of $i$-dimensional simplices of $X$, i.e., elements of $X$ with $(i+1)$ elements are denoted by $X_i.$ The dimension of the simplicial complex $X$ is defined to be the maximum dimension of a simplex in $X.$ 

Given a graph $G$, the clique complex of $G$, denoted $\text{Cl}(G),$ is the simplicial complex whose ground set is $V(G),$ and $F \subset V(G)$ is in $\text{Cl}(G)$ if and only if the vertices in $F$ form a clique in $G.$

\subsection{Operators on Simplicial Complexes}

Let $X$ be a $d$-dimensional simplicial complex on a totally ordered, finite ground set $V$ and let $i\in \{0,\ldots,d-1\}$. For $F\in X_i, K \in X_{i-1}$, we define the sign of $F$ over $K$ as:
\begin{gather*}
    [F:K] = \begin{cases}
        (-1)^j, & F=\{x_0,x_1,\ldots,x_i\}, x_0 < \dots < x_i, F\setminus K = \{x_j\},\\
        0, & \text{otherwise}.
    \end{cases}
\end{gather*}

The $i^{th}$ coboundary map $\delta_{i,X,\FF}$ is defined as the $M_{X_{i+1},X_i}(\FF)$ matrix:
\begin{gather*}
    (\delta_{i,X,\FF})_{F,K}=[F:K],
\end{gather*}
and the $i^{th}$ boundary matrix is defined as $\delta_{i,X,\FF}^T$. It is known that
\begin{gather*}
    \delta_{i+1, X, \FF} \circ \delta_{i, X, \FF}=0 \text{ and equivalently, } \delta_{i, X, \FF}^T \circ \delta_{i+1, X, \FF}^T=0, \quad \forall i\in \{0,\dots,d-1\}. 
\end{gather*}
For $i=1,\dots,d,$ we define the $i^{th}$ homology (group, module, vector space) and cohomology (group, module, vector space) by:
\begin{gather*}
    H_i(X,\FF)= \frac{\ker(\delta_{i-1, X, \FF}^T)}{\im(\delta_{i, X, \FF}^T)}, \quad H^i(X,\FF)= \frac{\ker(\delta_{i, X, \FF})}{\im(\delta_{i-1, X, \FF})},
\end{gather*}
respectively. These notions may be defined over any abelian groups or commutative rings, in particular, over $\Z.$ In this paper, we will not need this full generality. However, we will state several consequences of our developments to integral (co)boundary matrices and homology groups of clique complexes of graphs.  

When $G$ is a graph, we denote:
\begin{gather*}
    \delta_{i, G, \FF} = \delta_{i, \clc{G}, \FF}, 
\end{gather*}
for every $i$ for which these quantities are defined, and we will suppress the subscripts $G$ and/or $\FF$, and write $\delta_{i,\FF}^T$ or $\delta_{i}$ when there is no possibility of confusion.

\subsection{The Cycle Space of a Graph}

Following \cite[Definition 4.6]{Biggs}, we call $\kdt{0,R}{G}$ the cycle-(sub)space of $G$ over $R,$ or with coefficients in $R$. 
Note that in \cite[p.156]{cioaba-guo-ji-mim}, we inadvertently switched the definitions of cycle-subspace and cut-subspace. We also recall and extend the following definition from \cite[Definition 5.4]{cioaba-guo-ji-mim}.
\begin{defn}\label{def:TC}
Let $C=(v_1,\dots,v_{\ell})$ be an ordered tuple of vertices in $G$ such that $v_1,\dots,v_{\ell}$ form a cycle in $G$. Denoting $v_{\ell+1} = v_1$ and $e_i = \{v_i, v_{i+1}\}$ for $1 \leq i \leq \ell$, define
    \begin{gather}
        T(C) =  \sum_{i=1}^{\ell} c_i e_i \in R^{X_1},
    \end{gather}
    where the coefficients $c_1,\dots, c_{\ell}$ are defined recursively as follows:
    \begin{gather}\label{eq:coeffTC}
        c_i = \begin{cases}
            1, & \text{ if } i=1\\
            -c_{i-1}[e_{i-1} : v_i] \cdot [e_i:v_i], & \text{ if } 2\leq i \leq \ell.
        \end{cases}
\end{gather}
For simplicity, we will use $T(v_1,\ldots,v_{\ell})$ to denote $T(C)$ when necessary.
\end{defn}

In \cite[Lemma 5.5]{cioaba-guo-ji-mim}, we proved that $T(C) \in \ker(\delta_{0,R}^T)$ for $R=\re$. The same proof works verbatim for any ring $R$. In \cite{cioaba-guo-ji-mim} (see the discussion following Lemma 5.6), we also showed that for any graph $G$, $\ker(\delta_{0,R}^T)$ is generated by $\{T_R(C) : C \text{ an induced cycle in } G\}$ . Our argument there relied on Proposition 4.5 and Theorem 5.2 from \cite{Biggs} and also works verbatim for any ring $R$.  The following results were proved in \cite{cioaba-guo-ji-mim} for $\FF=\re,$ and the same proof is true for an arbitrary ring $R$.

\begin{lemma}\label{lemma:cycle-split-over-chord}
    Let $G$ be a graph, $C=(v_1,\dots,v_{\ell})$ a cycle of length $\ell \geq 4$, and $R$ a ring. Let $v_i,v_j$ be two non-consecutive vertices on $C$ such that $j > i$ and $v_i \sim v_j.$ Denote by $C_1$ and $C_2$ the cycles $C_1=(v_i,v_{i+1},\dots,v_j,v_i)$ and $C_2=(v_i,v_j,v_{j+1},\dots,v_{\ell},v_1,\dots,v_i).$ There are $c_1, c_2 \in \{1,-1\} \subset R\setminus \{0\}$ such that
    \begin{gather*}
        T_{R}(C) = c_1 T_{R}(C_1) + c_2 T_{R}(C_2). 
    \end{gather*}
\end{lemma}

\begin{lemma} \label{lemma:cycle-split-over-path}
    Let $R$ be a ring. Let $G$ be a graph, $C=(v_1,\dots,v_{\ell})$ a cycle of length $\ell \geq 4.$ Let $P=(v_1,w_2,\dots,w_k,v_i)$ be a path in $G$ such that $P \cap C =\{v_1,v_i\}.$ Also, denote by cycles 
    \begin{gather*}
        C_1 = (v_1,\dots,v_i,w_k,\dots, w_2), \quad C_2 = (v_1,w_2,\dots,w_k,v_i,\dots, v_{\ell}).
    \end{gather*}
    Then, there are constants $c_1,c_2 \in \{1,-1\}$ such that 
    \begin{gather*}
        T(C) = c_1 T(C_1) + c_2 T(C_2).
    \end{gather*}
\end{lemma}

\begin{prop}\label{prop:TC}
Let $R$ be a ring and $C=(v_1,\dots,v_{\ell})$ an ordered tuple of vertices that form a cycle in $G$ with edges $e_1=\{v_1,v_2\},\ldots, e_{\ell-1}=\{v_{\ell-1},v_{\ell}\}$, and $e_{\ell}=\{v_{\ell},v_1\}$. If $x \in \ker \delta_0^T$ has support $\{e_1,\dots,e_{\ell}\}$, then $x = c \cdot T_R(C)$ for some non-zero $c$.
\end{prop}

\begin{theorem} \label{thm:dim-cycle-space}
    If $G$ is a graph with $n$ vertices, $m$ edges, and $c$ connected components, then for any field $\FF,$
    \begin{gather*}
        \dkdt{0,\FF}{G} = m-n+c.
    \end{gather*}
    That is, $\dkdt{0,\FF}{G}$ is the same for every field $\FF.$
\end{theorem}
We will also use the following three lemmas. We proved them for $R=\re$ in \cite[Lemma 5.7]{cioaba-guo-ji-mim}, and the proof works for any ring $R.$

\begin{lemma}[Lemma 5.9, \cite{cioaba-guo-ji-mim}]\label{prop:wheel4}
Let $G$ be a graph with four vertices $a,b,c,d$ forming a cycle $(a,b,c,d)$. If $e$ is a common neighbor of $a,b,c$, and $d$, then there are $x_1,x_2,x_3,x_4 \in \{1, -1\} \subset R$ such that 
    \begin{gather}
        T_R(a,b,c,d) = \delta_{1,R}^T (x_1 \{a,b,e\} + x_2 \{b,c,e\}  + x_3 \{c,d,e\}  + x_4 \{a,d,e\}).
    \end{gather}
    Consequently, $T_R((a,b,c,d))) \in \im \delta_{1,R}^T.$ 
\end{lemma}

\begin{corollary}[Corollary 5.13, \cite{cioaba-guo-ji-mim}]\label{cor:TCinduced}
Let $\FF$ be a field. The elements of the set 
\begin{equation*}
    \{T_{\FF}(C): C=(v_1,\dots,v_{\ell}) \text{ a ordered tuple of vertices in } G \text{ such that } v_1,\dots,v_{\ell} \text{ is an induced cycle} \}
\end{equation*}   
span $\kdt{0,\FF}{G}$.
\end{corollary}

\begin{lemma}[Lemma 5.10, \cite{cioaba-guo-ji-mim}]\label{prop:supp_red}
Let $G$ be a graph and $R$ a ring. Assume that $C=(a,b,c,d,v_5,\dots,v_{\ell})$ is a cycle of length four or more, and $a,b,c,d$ have a common neighbor $e$. Denote $e_1 = \{a,b\}, e_2 = \{b,c\}, e_3 = \{c,d\}$, and $c_1 = [e_1, T_R(C)], c_2 = [e_2, T_R(C)], c_3 = [e_3, T_R(C)].$ There are $c'_4,c'_5,c'_6,x_1, x_2\in \{-1,1\}$ such that
    \begin{gather}
        c_1e_1 + c_2e_2 + c_3e_3 + \delta_{1,R}^T(c'_4 \{a,b,e\} + c'_5 \{b,c,e\} + c'_6 \{c,d,e\}) = x_1 \{a,e\} + x_2 \{d,e\}.
    \end{gather}
\end{lemma}

\begin{lemma} \label{lemma:two-cycle-join-vertex}
    Let $G$ be a graph, $C_1=(v_1,v_2,\dots,v_s)$ and $C_2=(v_1,v_2'\dots,v_t')$ two cycles in $G$ sharing exactly one vertex $v_1$. Let $R$ be a ring, and $\al \in \kdt{0,R}{G}$ whose support consists of precisely the edges in $C_1$ and $C_2$. There are constants $c_1, c_2 \in R \setminus \{0\}$ such that
    \begin{gather*}
        \al = c_1T(C_1) + c_2 T(C_2).
    \end{gather*}
    Moreover, if each edge (coordinate) in $\al$ has coefficient in $\{1,-1\},$ then $c_1, c_2$ are necessarily in $\{1,-1\}$ as well.
\end{lemma}

\begin{proof}
    For below, the indices are considered $\mod s$ or $\mod t,$ for the corresponding cycles. We prove the first statement first. Let:
    \begin{gather*}
        \al = \sum_{i=1}^s a_i \{v_i,v_{i+1}\} + \sum_{j=1}^t b_j \{v_j',v_{j+1}'\},
    \end{gather*}
    where the two sums aggregate the terms involving the edges of $C_1$ and $C_2,$ respectively. Consider $\beta := \sum_{i=1}^s a_i \{v_i,v_{i+1}\}.$ For $2 \leq i \leq s,$ the two edges in the support of $\al$ contributing to $v_i$ in the expansion of $\delta_0^T(\al)$ are both contained in the sum $\beta.$ Thus,
    \begin{gather*}
        0= [v_i, \delta_0^T(\al)] = [v_i, \delta_0^T(\beta)].
    \end{gather*}
    Thus, as seen in the proof of Lemma 5.5 in \cite{cioaba-guo-ji-mim}, the coefficients $a_2,\dots,a_s$ are all determined by the necessary conditions (following the same notations as in \cite{cioaba-guo-ji-mim}):
    \begin{gather*}
        a_i[\{v_{i}, v_{i+1}\} : v_i] = a_{i+1}[\{v_{i+1}, v_{i+2}\} : v_{i+1}].
    \end{gather*}
    Following the argument in the proof of the lemma in \cite{cioaba-guo-ji-mim} verbatim, these constraints imply
    \begin{gather*}
        \beta = \sum_{i=1}^s a_i \{v_i,v_{i+1}\} \in \ker \delta_0^T.
    \end{gather*}
    Thus, $\beta$ is a $\ker \delta_0^T$ element whose support edges are precisely the edges of the cycle $C_1$. By Lemma \ref{prop:TC}, there is some constant $c_1$ with $\beta = c_1 T(C_1).$ By a similar argument, we can show that there is some constant $c_2$ such that $\sum_{j=1}^t b_i \{v_j',v_{j+1}'\} = c_2T(C_2).$ Putting these together, we have
    \begin{gather*}
        \al = \sum_{i=1}^s a_i \{v_i,v_{i+1}\} + \sum_{j=1}^t b_i \{v_j',v_{j+1}'\} = c_1 T(C_1) + c_2 T(C_2).
    \end{gather*}
    This proves the first statement. Finally, if all the coefficients $a_i,b_j$ are in $\{1,-1\}$, then, since all coefficients in $T(C_1)$ and $T(C_2)$ are in $\{1,-1\}$ and they do not share any edge (nonzero-coordinate), we conclude that $c_1,c_2 \in \{1,-1\}.$ This proves the second statement.
\end{proof}

\subsection{Two Vanishing Theorems} \label{subsec:a-vanishing-theorem}

The following theorem is implicit in our previous paper \cite{cioaba-guo-ji-mim} (Theorem 5.4) for $R=\re$. We strengthen the version appearing in \cite{cioaba-guo-ji-mim} and show that the result holds over arbitrary fields.

\begin{theorem}\label{thm:4-vertex-common-neighbor}
If $G$ is a graph with the property that any induced cycle (of length at least four) has four consecutive vertices that have a common neighbor in $G$ and $\FF$ is a field, then the clique complex of $G$ satisfies $\kdt{0,\FF}{G}=\idt{1,\FF}{G}.$ Equivalently, $\hlcl{G}{\FF}=0.$
\end{theorem}

\begin{proof}
 It suffices to prove that if $C$ is a cycle in $G$, then $T_{\FF}(C) \in \idt{1,\FF}{G}.$ We prove this by induction on the length of $C.$ Let $C=(v_1,\dots,v_{\ell}).$ The claim is immediate for $\ell=3.$ If $\ell=4$ and $C$ is not induced, the claim follows by Lemma \ref{lemma:cycle-split-over-chord}. If $\ell=4$ and $C$ is induced, then there exists $v \in V(G)$, necessarily distinct from $v_i,i=1,2,3,4,$ such that $v \sim v_i, i=1,2,3,4.$ By Proposition \ref{prop:wheel4}, $T_{\FF}(C) \in \idt{1,\FF}{G}.$ Assume now $\ell=5$. If $C$ is not induced, the claim is again immediate by Lemma \ref{lemma:cycle-split-over-chord} and the discussions above. We now assume that $C$ is induced. Without loss of generality, we may also assume that there exists a vertex $v$ in $G$ such that $v \sim v_i,i=1,2,3,4.$ We consider two cases: $v=v_5$, and $v$ is not on the cycle $C.$ If $v=v_5,$ denote
    \begin{gather*}
        C_1=(v_1,v_2,v_5), \quad C_2=(v_2,v_3,v_5), \quad C_3=(v_3,v_4,v_5),
    \end{gather*}
    and,
    \begin{gather*}
        c_1 = -[\{v_1,v_2\}, T_{\FF}(C)]\cdot  [\{v_1,v_2\},T_{\FF}(C_1)], \quad c_3= -[\{v_3,v_4\},T_{\FF}(C)] \cdot [\{v_3,v_4\}, T_{\FF}(C_3)].
    \end{gather*}
    Denote
    \begin{gather*}
        X = T_{\FF}(C)+c_1T_{\FF}(C_1)+c_3T_{\FF}(C_3) \in \kdt{0,\FF}{G}.
    \end{gather*}
    By construction of $c_1$ and $c_3,$ the terms $\{v_1,v_2\}$ and $\{v_3,v_4\}$ vanish in $X$. Thus, the support edges of $X$ consists of the edges of $C_2,$ possibly along with the edges $\{v_1,v_5\}$ and $\{v_4,v_5\}.$ However, since $X \in \kdt{0,\FF}{G},$ the edges $\{v_1,v_5\}$ and $\{v_4,v_5\}$ cannot be in the support of $X$. Hence the $\kdt{0,\FF}{G}$ element $X$ has the same support as the triangle $C_2,$ and by Proposition \ref{prop:TC}, there exists $c_2 \in \FF$ such that
    \begin{gather*}
         T_{\FF}(C)+c_1T_{\FF}(C_1)+c_3T_{\FF}(C_3)=X=c_2T_{\FF}(C_2) \text{ or,}\\
         T_{\FF}(C) = -c_1T_{\FF}(C_1)-c_3T_{\FF}(C_3)+c_2T_{\FF}(C_2) \in \idt{1,\FF}{G}.
    \end{gather*}
    We now consider the case $v \not = v_5,$ i.e., $v$ is not on the cycle $C.$ Denote:
    \begin{gather*}
        C_1=(v,v_1,v_2), \quad C_2=(v,v_2,v_3), \quad C_3=(v,v_3,v_4), \quad C_4=(v_1,v,v_4,v_5).
    \end{gather*}
    By Proposition \ref{prop:supp_red}, there exists $c_1,c_2,c_3 \in \FF$ such that the support of the $\kdt{0,\FF}{G}$ element $X$ defined by
    \begin{gather*}
        X = T_{\FF}(C)+c_1T_{\FF}(C_1)+c_2T_{\FF}(C_2)+c_3T_{\FF}(C_3)
    \end{gather*}
    consists of the edges in the cycle $C_4.$ Thus, by Proposition \ref{prop:TC}, there exists $c_4 \in \FF$ such that $X=c_4 T_{\FF}(C_4)$. Furthermore, the length of the cycles $C_i,i=1,2,3,4$ is at most $4$, and by the discussions above, $T_{\FF}(C_i) \in \idt{1,\FF}{G}, i=1,2,3,4.$  Thus,
    \begin{gather*}
        T_{\FF}(C)=-c_1T_{\FF}(C_1)-c_2T_{\FF}(C_2)-c_3T_{\FF}(C_3)+X\\
        = -c_1T_{\FF}(C_1)-c_2T_{\FF}(C_2)-c_3T_{\FF}(C_3)+c_4T_{\FF}(C_4)\in \idt{1,\FF}{G}.
    \end{gather*}
    This proves the claim for all cycles of length at most $5.$
     
    We now assume $\ell \geq 6,$ and that the claim has been established for all cycles of length less than $\ell.$ Without loss of generality, there exists a vertex $v$ in $G$ such that $v \sim v_i,i=1,2,3,4.$ We will consider two cases based on whether $v$ is a vertex of $C$ or not.
    \begin{enumerate}[\text{Case}(1)]
        \item We consider the case $v$ not a vertex of $C$. In this case, as in the last case discussed above, we define 
        \begin{gather*}
        C_1=(v,v_1,v_2), \quad C_2=(v,v_2,v_3), \quad C_3=(v,v_3,v_4), \quad C_4=(v_1,v,v_4,\dots,v_{\ell}).
    \end{gather*}
    Similarly to the above, there exists $c_1,c_2,c_3,c_4 \in \FF$ such that
    \begin{gather*}
        T_{\FF}(C)=c_1T_{\FF}(C_1)+c_2T_{\FF}(C_2)+c_3T_{\FF}(C_3)+c_4T_{\FF}(C_4),
    \end{gather*}
    where the cycles $C_1,C_2,C_3,C_4$ have length at most $\ell -1.$ Thus, by the inductive hypothesis, $T_{\FF}(C)\in \idt{1,\FF}{G}.$

    \item We now consider the case $v=v_i$ for some $i$. We need  to consider three cases.
        \begin{enumerate}
            \item $v=v_5.$ Denote
            \begin{gather*}
                C_1=(v_1,v_2,v_5), \quad C_2=(v_2,v_3,v_5), \quad C_3=(v_3,v_4,v_5), \quad C_4=(v_5,\dots,v_{\ell},v_1),
            \end{gather*}
            and,
            \begin{gather*}
                c_1=-[\{v_1,v_2\}, T_{\FF}(C)] \cdot [\{v_1,v_2\}, T_{\FF}(C_1)], \quad c_2=-[\{v_2,v_3\}, T_{\FF}(C)] \cdot [\{v_2,v_3\}, T_{\FF}(C_2)],\\
                c_3=-[\{v_3,v_4\}, T_{\FF}(C)] \cdot [\{v_3,v_4\}, T_{\FF}(C_3)].
            \end{gather*}
            Also, let
            \begin{gather*}
                X = T_{\FF}(C)+c_1T_{\FF}(C_1)+c_2T_{\FF}(C_2)+c_3T_{\FF}(C_3) \in \kdt{0,\FF}{G}. 
            \end{gather*}
            By construction, the terms $\{v_1,v_2\},\{v_2,v_3\},$ and $\{v_3,v_4\}$ vanish in $X.$ Thus, the support of $X$ consists of all edges of the cycle $C_4,$ possibly except $\{v_1,v_5\}$, possibly along with the edges $\{v_2,v_5\},\{v_3,v_5\},$ and $\{v_4,v_5\}$. However, $X \in \kdt{0,\FF}{G}$ and $X \ne 0,$ every support edge of $X$ must be on a cycle all of whose edges are in the support of $X$. Therefore, the edge $\{v_1,v_5\}$ must be in the support of $X$, and the edges $\{v_2,v_5\},\{v_3,v_5\},$ and $\{v_4,v_5\}$ cannot be in the support of $X.$ Hence, the support of $X$ is precisely the set of edges in the cycle $C_4.$ Thus, there exists $c_4 \in \FF$ such that $X=c_4T_{\FF}(C_4).$ Consequently,
            \begin{gather*}
                T_{\FF}(C)=-c_1T_{\FF}(C_1)-c_2T_{\FF}(C_2)-c_3T_{\FF}(C_3)+c_4T_{\FF}(C_4).
            \end{gather*}
            However, the cycles $C_i,i=1,2,3,4$ have lengths less than $\ell,$ and by inductive hypothesis, $T_{\FF}(C_i)\in \idt{1,\FF}{G}, i=1,2,3,4.$ Consequently, $T_{\FF}(C) \in \idt{1,\FF}{G}$ as well.
            \item The proof in the case $v=v_{\ell}$ is almost identical to the case $v=v_5,$ and will be skipped.
            \item We now consider the case $v=v_i$ for some $6 \leq i \leq \ell-1.$ Denote
            \begin{gather*}
                C_1=(v_1,v_2,v_i), \quad C_2=(v_2,v_3,v_i), \quad C_3=(v_3,v_4,v_i),\\
                C_4=(v_1,v_i,\dots,v_{\ell}), \quad C_5=(v_4,\dots,v_i).
            \end{gather*}
        \end{enumerate}
        By Proposition \ref{prop:supp_red}, there exist $c_1,c_2,c_3 \in \FF$ such that the $\kdt{0,\FF}{G}$ element $X$ defined by:
        \begin{gather*}
            X = T_{\FF}(C)+c_1T_{\FF}(C_1)+c_2T_{\FF}(C_2)+c_3T_{\FF}(C_3)
        \end{gather*}
        has support the edges of the cycles $C_4$ and $C_5$. By Lemma \ref{lemma:two-cycle-join-vertex}, there exists $c_4,c_5 \in \FF$ such that:
        \begin{gather*}
            X=c_4T_{\FF}(C_4)+c_5T_{\FF}(C_5).
        \end{gather*}
        Thus,
        \begin{gather*}
            T_{\FF}(C)=\sum_{i=1}^3(-c_i)T_{\FF}(C_i)+\sum_{i=4}^5c_iT_{\FF}(C_i).
        \end{gather*}
        However, the cycles $C_i,i=1,\dots,5$ have lengths strictly less than $\ell,$ and by the inductive hypothesis, each $T_{\FF}(C_i)\in \idt{1,\FF}{G},$ and consequently, $T_{\FF}(C) \in \idt{1,\FF}{G}.$
    \end{enumerate}
    By induction, we conclude that the claim is true for every $\ell \geq 3,$ and the theorem follows.
\end{proof}

\begin{theorem}\label{thm:4-vertex-common-neighbor-srg}
If $G$ is a strongly regular graph with the property that any induced cycle of length $4$ or $5$ has four vertices that have a common neighbor in $G$ and $\FF$ is a field, then the clique complex of $G$ satisfies $\kdt{0,\FF}{G}=\idt{1,\FF}{G}.$ Equivalently, $\hlcl{G}{\FF}=0.$
\end{theorem}

\begin{proof}
 In the following, $T$ denotes $T_{\FF}$. It suffices to prove that if $C$ is a cycle in $G$, then $T_{\FF}(C) \in \idt{1,\FF}{G}.$ We prove this by induction on the length of $C.$ Let $C=(v_1,\dots,v_{\ell}).$ Using an identical argument as in the proof of Theorem \ref{thm:4-vertex-common-neighbor}, we can show that the claim is true for $\ell=3,4,5.$

We now assume $\ell \geq 6,$ and that the claim has been established for all cycles of length less than $\ell.$ If $C$ is not induced, we may assume without loss of generality that $v_1 \sim v_i$ for some $3 \leq i \leq \ell-1.$ Let $C_1=(v_1,\dots,v_i)$ and $C_2=(v_1,v_i,\dots,v_{\ell}).$ By Lemma \ref{lemma:cycle-split-over-chord}, there are constants $c_1,c_2$ such that $T(C)=c_1T(C_1)+c_2T(C_2).$ Both the cycles $C_1$ and $C_2$ have lengths strictly less than $\ell,$ thus, by the induction hypothesis, $T(C_1) \in \idt{1,\FF}{G}$ and $T(C_2) \in \idt{1,\FF}{G}.$ Therefore, $T(C) \in \idt{1, \FF}{G}$

Consider now the case that $C$ is induced. In particular, $v_1 \sim v_2,\ v_2 \sim v_3,\ v_1 \not \sim v_3.$ Therefore, the $\mu$ parameter of $G$ is positive. Since $v_1 \not \sim v_4$, there is some vertex $x$ in $G$, necessarily not on $C$, such that $v_1 \sim x$ and $v_4 \sim x.$ Denote $C_1=(v_1,v_2,v_3,v_4,x)$ and $C_2=(v_1,x,v_4,\dots,v_{\ell}).$ By Lemma \ref{lemma:cycle-split-over-path}, there are constants $c_1,c_2$ such that $T(C)=c_1T(C_1)+c_2T(C_2).$ Both the cycles $C_1$ and $C_2$ have lengths strictly less than $\ell,$ thus, by the induction hypothesis, $T(C_1) \in \idt{1,\FF}{G}$ and $T(C_2) \in \idt{1,\FF}{G}.$ Therefore, $T(C) \in \idt{1, \FF}{G}$ in this case also, Thus, $T(C) \in \idt{1,\FF}{G}$ in all cases, and the theorem is proved.
\end{proof}

\section{Latin Square Graphs} \label{sec:latin-square-graphs}

In this section, we explicitly compute all homologies of the clique complexes of the strongly regular graphs associated with a Latin square of order at least $5$. We will show that all homologies except $H_2$ vanish over all fields, and $H_2$ has connections with the number $2 \times 2$ Latin subsquares or intercalates. 

\subsection{Latin Squares and Latin Square Graphs}
For a positive integer $n$, a Latin square $M$ of order $n$ is an $n \times n$ matrix with symbols $\{1,\dots,n\}$ such that each symbol in $\{1,\dots,n\}$ appears exactly once in each row and in each column. %Given a Latin square $M$, we define the symbol function $s_M: [n] \times [n] \to [n]$ so that $s_M(i,j)$ denotes the symbol in the $(i,j)$ position of the Latin square $M$, for $(i,j) \in [n] \times [n].$ 

We define the Latin square graph $L(M)$ to be the graph with vertex set $[n] \times [n]$ (corresponding to the matrix coordinate positions of $M$), where for $(i,j), (x,y) \in [n] \times [n], (i,j) \not = (x,y)$,
\begin{equation*}
(i,j) \sim (x,y) \iff i=x \text{ or } j = y \text{ or } M(i,j)=M(x,y).  
\end{equation*}
We partition the edges of $G:=L(M)$ into three sets.
\begin{defn}
    Let $G$ be the Latin square graph of the $n \times n$ Latin square $M$, and $(i,j)$, $(x,y)$ are two adjacent vertices of $G,$ and $e$ is the edge containing the two vertices. We say:
    \begin{gather*}
        e \text{ is horizontal if } i = x; \text{ vertical if } j = y, \text{ and, } \text{ diagonal if } M(i,j) = M(x,y).
    \end{gather*}
\end{defn}
%For $v \in [n] \times [n],$ we will also denote by $M[v]$ the symbol at position $v$ in $M,$ following the usual matrix notation.

\begin{defn}
A cycle $C$ of $G$ is called rectangular if it contains no diagonal edges.

%If $C$ is a cycle in $G$ with no diagonal edges, then $C$ will be said to be a rectangular cycle.
\end{defn}
Note that if $n \geq 2$, then $G$ has at least one edge of each type. If $M$ is a Latin square of order $n \geq 4$, then the Latin square graph $L(M)$ is strongly regular with parameters $(n^2, 3(n-1), n, 6)$.

\subsection{The First Homology of Clique Complex of Latin Square Graphs}
\label{subsec:the-h1-of-clique-complex-of-latin-square-grpahs}

In this section, we will show that if $M$ is a Latin square of order $n \geq 4$ and $G$ the strongly regular graph associated to $M$, then $\hlcl{G}{\FF}$ vanishes over every field. We prove this by showing that if $C$ is any cycle in $G$, then $T_{\FF}(C) \in \idt{1,\FF}{G}$. In that regard, we first show that this is true if $C$ has length four and contains only non-diagonal edges in Proposition \ref{prop:rectangular-4-cycle}. In Proposition \ref{prop:rectangular-cycles}, we generalize this to all cycles with only non-diagonal edges by showing that such cycles may be inductively {\em homotoped} to a cycle  of the same kind with a smaller number of edges. Finally, we complete the proof that $\hlcl{G}{\FF}=0$ in Theorem \ref{vanishing-latin-h1} by showing that cycles with diagonal edges may be {\em homotoped} to cycles with fewer diagonal edges.

\subsubsection{The Rectangular Cycles}

For the remainder of the subsection, we will assume $M$ is a Latin square of order $n \geq 4$ and $G=L(M)$ the strongly regular graph associated to $M.$

\begin{prop} \label{prop:rectangular-4-cycle}
Let $C$ be rectangular $4$-cycle in $G$. Then $T_{\Z}(C) \in \idt{1,\Z}{G}$. Moreover, there is $s\in \mathbb{N}$ and there are triangles $C_1,\ldots,C_s$ in $G$ and coefficients $c_1,\ldots,c_s\in \{-1,1\}$ such that
\begin{gather*}
    T_{\Z}(C) = \sum_{i=1}^s c_i\delta_{1,\Z}^T(C_i).
\end{gather*}
\end{prop}

\begin{proof}
By an application of Lemma \ref{lemma:cycle-split-over-chord}, the proposition is true if $C$ is not induced, so we may assume that $C$ is induced. In particular, $C$ does not have three vertices with the same $x$ coordinate or the same $y$ coordinate in the matrix coordinate system, nor does it have two vertices with the same symbol. Thus, the four vertex coordinates form a rectangle in the plane. Let the vertices of the cycle be $a=(x,y), b=(x,z), c=(t,z),$ and $d=(t,y),$ enumerated clockwise, starting for the top-left corner, where the ordered pairs refer to matrix coordinates.

We consider two other vertices, distinct from $a,b,c,d$, described as follows:
\begin{enumerate}
    \item vertex $b'$ on row $t$ and not on columns $y$ or $z$ such that: $M[b']=M[b],$ and,
    \item vertex $b''$ on column $y$ and not on rows $x$ or $t$ such that: $M[b'']=M[b].$
\end{enumerate}
Because $C$ is an induced cycle, $b'\neq d$ and $b''\neq d$.

We consider the induced subgraph on these $6$ vertices. First, assume the  ordering on these vertices: $a < b < c < d < b' < b''$. 

It is straightforward to verify the following relation:
\begin{gather} \label{eq:T-z(a,b,c,d)}
    T_{\Z}(a,b,c,d)=\{a,b\} + \{b,c\} + \{c,d\} - \{a,d\}\\
    =\delta_{1,\Z}^T(\{a,b,b''\}-\{a,d,b''\} + \{b,c,b'\} + \{b,b',b''\} + \{c,d,b'\} - \{d,b',b''\}) \in \im \delta_{1,\Z}^T.
\end{gather}
If $\delta_{1,\Z}'$ is the coboundary matrix under a different ordering of the vertices in $G$, then there are diagonal and square integer matrices $S_1,S_2$ (of appropriate sizes) whose diagonal entries are $\pm 1$ such that (see \cite[Appendix A]{cioaba-guo-ji-mim}):
\begin{gather} \label{eq:eq-2}
    \delta_{1,\Z}' = S_1\delta_{1,\Z}S_2.
    % \text{ and } S_1^TS_1= I, S_2^TS_2=I.
\end{gather}

If $X=\clc{G}$, then let $v = T_{\Z}(a,b,c,d) \in \Z^{X_1}$ and $w \in \Z^{X_2}$ be the vectors in equation \eqref{eq:T-z(a,b,c,d)} above, such that $\delta_{1,\Z}^T w = v$. We have that
\begin{gather}
    \delta_{1,\Z}^{'T}(S_1w) = (\delta_{1,\Z}^{'T}S_1)w=(S_2\delta_{1,\Z}^T)w=S_2(\delta_{1,\Z}^Tw) = S_2v.
\end{gather}
Because $S_1$ is a diagonal matrix with diagonal entries $\pm 1$, $S_1w$ has the same support as $w,$ with $(S_1w)_i = \pm w_i$ for each coordinate $i$. Hence, there are $c_1,\dots,c_6 \in \{1, -1\}$ such that:
\begin{gather*}
    S_2v=\delta_{1,\Z}^{'T}(S_1w)\\
    =\delta_{1,\Z}'^T(c_1\{a,b,b''\}+c_2\{a,d,b''\} +c_3 \{b,c,b'\} + c_4\{b,b',b''\} + c_5\{c,d,b'\} +c_6 \{d,b',b''\}).
\end{gather*}
In particular, $S_2v \in \im(\delta_{1,\Z}'^T).$ If $\delta_0'$ is the new coboundary operator under the ordering, then $S_2v \in \ker \delta_0^{'T}.$ Moreover, the support of $S_2v$ is the same as the support of $v$, namely the set of the edges of the cycle $(a,b,c,d),$ and all nonzero coefficients (entries) in $S_2v$ are $\pm 1$. This implies that $S_2v = \pm T_{\Z}(a,b,c,d)$, where the last term has been computed with respect to the new ordering. Thus, $T_{\Z}(a,b,c,d)=\pm S_2v$. 

If $T_{\Z}(a,b,c,d)=S_2v,$ then $T_{\Z}(a,b,c,d)=S_2v\in \im(\delta_{1,\Z}'^T).$ If $T(a,b,c,d)=-S_2v,$ then
\begin{gather*}
T(a,b,c,d)=\delta_{1,\Z}'^T(-c_1\{a,b,b''\}-c_2\{a,d,b''\}-c_3 \{b,c,b'\} -c_4\{b,b',b''\} - c_5\{c,d,b'\} - c_6 \{d,b',b''\}),
\end{gather*}
which implies that $T(a,b,c,d) \in \im(\delta_{1,\Z}'^T).$ In each case, $T_{\Z}(C) \in \im(\delta_{1,\Z}'^T)$ and the proof follows.
\end{proof}

\begin{corollary} \label{corollary:rectangular-4-cycle-field}
Let $\FF$ be a field. If $C$ is a rectangular $4$-cycle in $G$, then $T_{\FF}(C) \in \idt{1,\FF}{G}$. Moreover, there is $s\in \mathbb{N}$, triangles $C_1,\ldots,C_s$ in $G$, and $c_1,\ldots,c_s \in \{-1,1\}\subset \FF$ such that
\begin{gather*}
    T_{\FF}(C) = \sum_{i=1}^s c_i\delta_{1,\FF}^T(C_i).
\end{gather*}
\end{corollary}
\begin{proof}
From Proposition \ref{prop:rectangular-4-cycle}, there is $s\in \mathbb{N}$, triangles  $C_1,\ldots,C_s$ in $G$ and $c_1,\ldots,c_s \in \{-1,1\} \subset \Z$ such that
\begin{gather*}
    T_{\Z}(C) = \sum_{i=1}^s c_i\delta_{1,\Z}^T(C_i).
\end{gather*}
If $\fchar{\FF}=0,$ then $\Q$ embeds in $\FF$ as the prime field of $\FF,$ and the above relation holds over $\FF$ as well. On the other hand, if $\fchar{\FF} > 0,$ let $p = \fchar{\FF}.$ Then, taking $\pmod p$ on both sides of the above equation, we have:
\begin{gather}
    T_{\FF_p}(C) = \sum_{i=1}^s c_i\delta_{1,\FF_p}^T(C_i).
\end{gather}
Note that, here we have used the observation $T_{\Z}(C) \pmod p = T_{\FF_p}(C).$ Since $\FF_p \subset \FF,$ the above equation also holds in $\FF,$ and so,
\begin{gather*}
     T_{\FF}(C) = \sum_{i=1}^s c_i\delta_{1,\FF}^T(C_i).
\end{gather*}
\end{proof}

\begin{prop} \label{prop:rectangular-cycles}
    Let $C$ be a rectangular cycle in $G$ and $\FF$ a field. Then, $T_{\FF}(C) \in \idt{1,\FF}{G}.$ 
\end{prop}
\begin{proof}
    In this proof, we will denote $\delta_{0,\FF}, \delta_{1,\FF}, T_{\FF}$ by $\delta_0, \delta_1, T,$ respectively. 
    
    The proof is by strong induction on the number $s$ of edges in $C$. The base case is $s=3,$ when $C$ is a triangle in $G$, and $T(C)=\pm \delta_{1}^T(C)\in \idt{1,\FF}{G}$. If $s=4$, by Proposition \ref{corollary:rectangular-4-cycle-field}, $T(C) \in \idt{1}{G}.$ Assume now that the proposition holds for all rectangular cycles of length strictly less than $s,$ and fix cycle $C$ of length $s,$ where $s \geq 5.$ We will consider several cases and show that in each case, the inductive hypothesis applies, and $T(C) \in \idt{1,\FF}{G}.$

First, consider the case when $C=(v_1,\dots,v_s)$ has two vertices $v_i, v_j, i < j,$ that are adjacent in $G$, but the edge $\{v_i,v_j\}$ is not in $C$. 

Denote $C_1=(v_i,v_{i+1},\dots,v_j)$ and $C_2=(v_i,v_j,v_{j+1},\dots,v_s,v_1,\dots,v_i).$ By Lemma \ref{lemma:cycle-split-over-chord}, there are constants $c_1,c_2 \in \FF^*$ such that:
\begin{gather*}
        T(C) = c_1 T(C_1) + c_2 T(C_2),
\end{gather*}
where both $C_1$ and $C_2$ have only horizontal or vertical edges and have lengths strictly less than $s$. By the inductive hypothesis, $c_1 T(C_1), c_2 T(C_2) \in \idt{1}{G},$ and so, $T(C) \in \idt{1}{G}$ in this case.

In particular, if $C$ has three vertices (not necessarily consecutive) with the same first coordinate or the same second coordinate, the inductive hypothesis applies and we have $T(C) \in \idt{1}{G}.$ 

Assume now that $C$ does not contain three vertices (not necessarily consecutive) with the same first coordinate or the same second coordinate. In particular, $C$ does not contain two consecutive edges that are both horizontal or both vertical. Since $C$ does not contain any diagonal edge, the edges in $C$ alternate between vertical and horizontal types.

Consider the first four consecutive vertices of the cycle: 
\begin{equation*}
    v_1=(x_1,y_1),\, v_2=(x_2,y_2),\, v_3=(x_3,y_3),\, v_4=(x_4,y_4),
\end{equation*}
where the coordinate values refer to positions in the matrix $M$. 

Without loss of generality, assume that the edges $\{v_1,v_2\}, \{v_3,v_4\}$ are horizontal, $\{v_2,v_3\}$ is vertical, and $y_1<y_2$. Therefore, $y_2 = y_3$ and $y_4 \not = y_2.$
    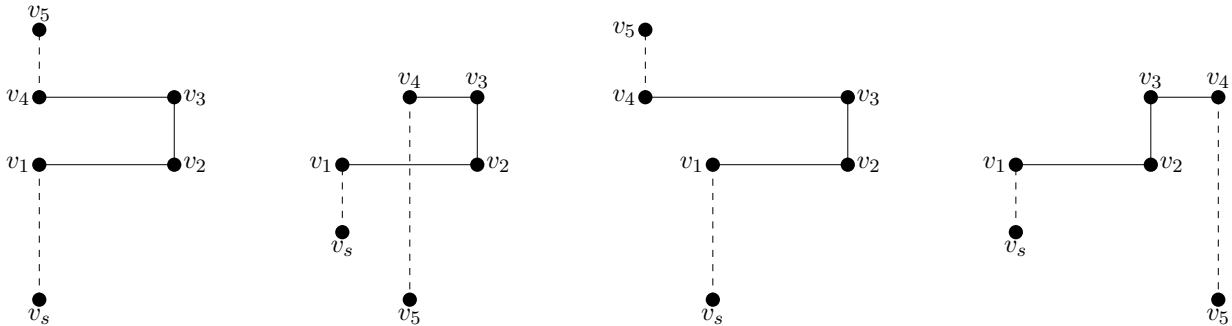
\begin{figure}[h]
        \centering
        \resizebox{\textwidth}{!}{
            % Rectangular cases
% \includegraphics[width=0.5\linewidth]{}
\begin{tikzpicture}
    \coordinate (v_1) at (0,0);
    \coordinate (v_2) at (2,0);
    \coordinate (v_3) at (2,1);
    \coordinate (v_4) at (0,1);
    \coordinate (v_5) at (0,2);
    \coordinate (v_s) at (0,-2);
    \fill (v_1) circle(3pt) node[anchor=east]{$v_1$};
    \fill (v_2) circle(3pt) node[anchor=west]{$v_2$};
    \fill (v_3) circle(3pt) node[anchor=west]{$v_3$};
    \fill (v_4) circle(3pt) node[anchor=east]{$v_4$};
    \fill (v_5) circle(3pt) node[anchor=south]{$v_5$};
    \fill (v_s) circle(3pt) node[anchor=north]{$v_s$};
    \draw (v_s) edge[dashed] (v_1);
    \draw (v_4) edge[dashed] (v_5);
    \draw (v_1) -- (v_2) -- (v_3) -- (v_4);
\end{tikzpicture}
\hspace{10mm}
\begin{tikzpicture}
    \coordinate (v_1) at (0,0);
    \coordinate (v_2) at (2,0);
    \coordinate (v_3) at (2,1);
    \coordinate (v_4) at (1,1);
    \coordinate (v_5) at (1,-2);
    \coordinate (v_s) at (0,-1);
    \fill (v_1) circle(3pt) node[anchor=east]{$v_1$};
    \fill (v_2) circle(3pt) node[anchor=west]{$v_2$};
    \fill (v_3) circle(3pt) node[anchor=south]{$v_3$};
    \fill (v_4) circle(3pt) node[anchor=south]{$v_4$};
    \fill (v_5) circle(3pt) node[anchor=north]{$v_5$};
    \fill (v_s) circle(3pt) node[anchor=north]{$v_s$};
    \draw (v_s) edge[dashed] (v_1);
    \draw (v_4) edge[dashed] (v_5);
    \draw (v_1) -- (v_2) -- (v_3) -- (v_4);
\end{tikzpicture}
\hspace{10mm}
\begin{tikzpicture}
    \coordinate (v_1) at (0,0);
    \coordinate (v_2) at (2,0);
    \coordinate (v_3) at (2,1);
    \coordinate (v_4) at (-1,1);
    \coordinate (v_5) at (-1,2);
    \coordinate (v_s) at (0,-2);
    \fill (v_1) circle(3pt) node[anchor=east]{$v_1$};
    \fill (v_2) circle(3pt) node[anchor=west]{$v_2$};
    \fill (v_3) circle(3pt) node[anchor=west]{$v_3$};
    \fill (v_4) circle(3pt) node[anchor=east]{$v_4$};
    \fill (v_5) circle(3pt) node[anchor=east]{$v_5$};
    \fill (v_s) circle(3pt) node[anchor=north]{$v_s$};
    \draw (v_s) edge[dashed] (v_1);
    \draw (v_4) edge[dashed] (v_5);
    \draw (v_1) -- (v_2) -- (v_3) -- (v_4);
\end{tikzpicture}
\hspace{10mm}
\begin{tikzpicture}
    \coordinate (v_1) at (0,0);
    \coordinate (v_2) at (2,0);
    \coordinate (v_3) at (2,1);
    \coordinate (v_4) at (3,1);
    \coordinate (v_5) at (3,-2);
    \coordinate (v_s) at (0,-1);
    \fill (v_1) circle(3pt) node[anchor=east]{$v_1$};
    \fill (v_2) circle(3pt) node[anchor=west]{$v_2$};
    \fill (v_3) circle(3pt) node[anchor=south]{$v_3$};
    \fill (v_4) circle(3pt) node[anchor=south]{$v_4$};
    \fill (v_5) circle(3pt) node[anchor=north]{$v_5$};
    \fill (v_s) circle(3pt) node[anchor=north]{$v_s$};
    \draw (v_s) edge[dashed] (v_1);
    \draw (v_4) edge[dashed] (v_5);
    \draw (v_1) -- (v_2) -- (v_3) -- (v_4);
\end{tikzpicture}
        }
        \caption{From the left: the cases $y_1=y_4,$ $y_1 < y_4 < y_2$, $y_4 < y_1$, $y_4 > y_2,$ respectively.}
        \label{fig:rectangular-cases}
    \end{figure}
    Since $s \geq 5,$ $v_1$ and $v_4$ are non-consecutive vertices on $C$. 
    
    We consider several cases based on the relative position of $y_4$ with respect to $y_1$ and $y_2$, see Figure \ref{fig:rectangular-cases}. If $y_4=y_1$, then $v_1$ and $v_4$ are non-consecutive vertices on $C$ that have the same second coordinate. By our hypothesis, this case is not possible.

    Consider now the case $y_1 < y_4 < y_2$, see Figure \ref{fig:rectangular-cases-1}. By our hypothesis, the vertex $v=(x_1,y_4)$ of $G$ is not on $C$ since otherwise, $v,v_1,v_2$ would be three distinct vertices on $C$ with the same first coordinate. Because $v,v_4,v_5$ have the same first coordinates, $C_1:=\{v_4,v,v_5\}$ is a triangle in $G$, and $C_1$ shares exactly one edge $\{v_4,v_5\}$ with $C.$ Furthermore, since $v_1,v,v_2$ have the same second coordinates, $C_2=\{v_1,v,v_2\}$ is a triangle in $G$ that shares exactly one edge $\{v_1,v_2\}$ with $C$. Hence, we may choose $c_1,c_2 \in \{1,-1\}$ such that:
    \begin{gather*}
        \al = T(C)+ c_1 \delta_1^T(C_1) + c_2 \delta_1^T(C_2) \in \ker \delta_0^T,
    \end{gather*}
    and the terms $\{v_4,v_5\}$ and $\{v_1,v_2\}$ vanish in $\al$. Therefore, the support of the $\kdt{0}{G}$ element $\al$ consists of precisely the edges in the cycles  $C'=(v,v_5,\dots,v_s,v_1)$ and $C''=(v,v_2,v_3,v_4)$, which share exactly one vertex, namely $v$. By Lemma \ref{lemma:two-cycle-join-vertex}, there are $c',c'' \in \FF$ such that
    \begin{gather*}
        \al = c'T(C')+c''T(C'').
    \end{gather*}
    Both cycles $C'$ and $C''$ contain only vertical and horizontal edges, and their lengths are strictly less than $s.$ By the induction hypothesis, $T(C'), T(C'') \in \idt{1}{G}$ and consequently, $\alpha\in \idt{1}{G}$. Finally, as $C_1$ and $C_2$ are triangles, $T(C_1),T(C_2) \in \idt{1}{G}$. We conclude that
    \begin{gather*}
        T(C)=\al-c_1 \delta_1^T(C_1)-c_2 \delta_1^T(C_2)\in \idt{1}{G}.
    \end{gather*}
    This proves the case $y_1 < y_4 < y_2.$
    
    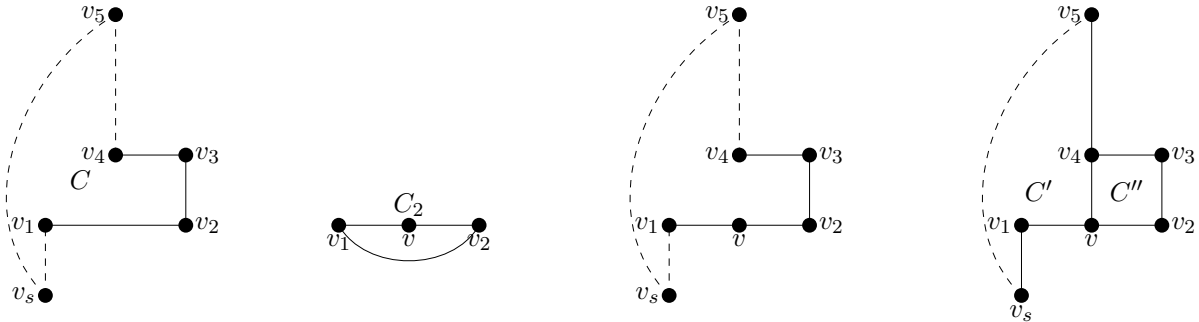
\begin{figure}[h]
        \centering
        \resizebox{\textwidth}{!}{
                \begin{tikzpicture}[baseline=(v_4.base)]
        \coordinate (v_1) at (0,0);
        \coordinate (v_2) at (2,0);
        \coordinate (v_3) at (2,1);
        \coordinate (v_4) at (1,1);
        \coordinate (v_5) at (1,3);
        \coordinate (v_s) at (0,-1);
        \fill (v_1) circle(3pt) node[anchor=east]{$v_1$};
        \fill (v_2) circle(3pt) node[anchor=west]{$v_2$};
        \fill (v_3) circle(3pt) node[anchor=west]{$v_3$};
        \fill (v_4) circle(3pt) node[anchor=east]{$v_4$};
        \fill (v_5) circle(3pt) node[anchor=east]{$v_5$};
        \fill (v_s) circle(3pt) node[anchor=east]{$v_s$};
        \draw (v_s) edge[dashed] (v_1);
        \draw (v_4) edge[dashed] (v_5);
        \draw (v_1) -- (v_2) -- (v_3) -- (v_4);
        \draw (v_s) edge[dashed,out=135,in=210] (v_5);
        \filldraw (.5,0.65) circle(0pt) node[anchor=center]{$C$};
    \end{tikzpicture}
    \hspace{10mm}
    \begin{tikzpicture}[baseline=(v_4.base)]
        \coordinate (v_1) at (0,0);
        \coordinate (v_2) at (2,0);
        \coordinate (v) at (1,0);
        \fill (v_1) circle(3pt) node[anchor=north]{$v_1$};
        \fill (v_2) circle(3pt) node[anchor=north]{$v_2$};
        \fill (v) circle(3pt) node[anchor=north]{$v$};
        \draw (v_1) -- (v) -- (v_2);
        \draw (v_1) edge[out=-60, in=-120] (v_2);
        \filldraw (1,0) circle(0pt) node[anchor=south]{$C_2$};
    \end{tikzpicture}
    \hspace{10mm}
    \begin{tikzpicture}[baseline=(v_4.base)]
        \coordinate (v_1) at (0,0);
        \coordinate (v_2) at (2,0);
        \coordinate (v_3) at (2,1);
        \coordinate (v_4) at (1,1);
        \coordinate (v_5) at (1,3);
        \coordinate (v_s) at (0,-1);
        \coordinate (v) at (1,0);
        \fill (v_1) circle(3pt) node[anchor=east]{$v_1$};
        \fill (v_2) circle(3pt) node[anchor=west]{$v_2$};
        \fill (v_3) circle(3pt) node[anchor=west]{$v_3$};
        \fill (v_4) circle(3pt) node[anchor=east]{$v_4$};
        \fill (v_5) circle(3pt) node[anchor=east]{$v_5$};
        \fill (v_s) circle(3pt) node[anchor=east]{$v_s$};
        \fill (v) circle(3pt) node[anchor=north]{$v$};
        \draw (v_s) edge[dashed] (v_1);
        \draw (v_4) edge[dashed] (v_5);
        \draw (v_1) -- (v_2) -- (v_3) -- (v_4);
        \draw (v_s) edge[dashed,out=135,in=210] (v_5);
        % \filldraw (1,0.25) circle(0pt) node[anchor=south]{$C'$};
    \end{tikzpicture}
    \hspace{10mm}
    \begin{tikzpicture}[baseline=(v_4.base)]
        \coordinate (v_1) at (0,0);
        \coordinate (v_2) at (2,0);
        \coordinate (v_3) at (2,1);
        \coordinate (v_4) at (1,1);
        \coordinate (v_5) at (1,3);
        \coordinate (v_s) at (0,-1);
        \coordinate (v) at (1,0);
        \fill (v_1) circle(3pt) node[anchor=east]{$v_1$};
        \fill (v_2) circle(3pt) node[anchor=west]{$v_2$};
        \fill (v_3) circle(3pt) node[anchor=west]{$v_3$};
        \fill (v_4) circle(3pt) node[anchor=east]{$v_4$};
        \fill (v_5) circle(3pt) node[anchor=east]{$v_5$};
        \fill (v_s) circle(3pt) node[anchor=north]{$v_s$};
        \fill (v) circle(3pt) node[anchor=north]{$v$};
        \draw (v) -- (v_2) -- (v_3) -- (v_3) -- (v_4) -- (v) -- (v_1) -- (v_s);
        \draw (v_5) -- (v_4);
        \draw (v_s) edge[dashed,out=135,in=210] (v_5);
        \filldraw (1.5,0.5) circle(0pt) node[anchor=center]{$C''$};
        \filldraw (0.25,0.5) circle(0pt) node[anchor=center]{$C'$};
    \end{tikzpicture}
        }
        \caption{The case $y_1 < y_4 < y_2.$ Curved, dashed segments represent arbitrary length paths in corresponding cycles. From the left: cycles $C,C_1,C',$ and together, $C_2$ and $C_3,$ respectively.}
        \label{fig:rectangular-cases-1}
    \end{figure}

    Consider the case $y_4 < y_1$, see Figure \ref{fig:rectangular-cases-2}. We may use a similar argument as in the previous case, with vertex $v=(x_4,y_1)$, triangles $C_1=\{v_4,v,v_3\}, C_2=\{v,v_1,v_s\},$ and cycles $C''=(v,v_1,v_2,v_3)$ and $C'=(v,v_4,\dots,v_s)$ to show that $T(C) \in \idt{1}{G}$ in this case as well. 
    \begin{figure}[htbp!]
        \centering
        \resizebox{\textwidth}{!}{
                \begin{tikzpicture}[baseline=(v_4.base)]
        \coordinate (v_1) at (0,0);
        \coordinate (v_2) at (2,0);
        \coordinate (v_3) at (2,1);
        \coordinate (v_4) at (-1,1);
        \coordinate (v_5) at (-1,2);
        \coordinate (v_s) at (0,-2);
        \fill (v_1) circle(3pt) node[anchor=east]{$v_1$};
        \fill (v_2) circle(3pt) node[anchor=west]{$v_2$};
        \fill (v_3) circle(3pt) node[anchor=west]{$v_3$};
        \fill (v_4) circle(3pt) node[anchor=east]{$v_4$};
        \fill (v_5) circle(3pt) node[anchor=east]{$v_5$};
        \fill (v_s) circle(3pt) node[anchor=north]{$v_s$};
        \draw (v_s) edge[dashed] (v_1);
        \draw (v_4) edge[dashed] (v_5);
        \draw (v_1) -- (v_2) -- (v_3) -- (v_4);
        \draw (v_s) edge[dashed, out=150,in=210] (v_5);
        \filldraw (.5,0.65) circle(0pt) node[anchor=center]{$C$};
    \end{tikzpicture}
    
    \begin{tikzpicture}[baseline=(v_4.base)]
        \coordinate (v_3) at (2,1);
        \coordinate (v_4) at (-1,1);
        \coordinate (v) at (0,1);
        \fill (v_3) circle(3pt) node[anchor=west]{$v_3$};
        \fill (v_4) circle(3pt) node[anchor=east]{$v_4$};
        \fill (v) circle(3pt) node[anchor=north]{$v$};
        \draw (v_4) -- (v) -- (v_3);
        \draw (v_4) edge[out=-60,in=-120] (v_3);
        \filldraw (.5,0.65) circle(0pt) node[anchor=center]{$C_1$};
    \end{tikzpicture}
    \begin{tikzpicture}[baseline=(v_4.base)]
        \coordinate (v_1) at (0,0);
        \coordinate (v_2) at (2,0);
        \coordinate (v_3) at (2,1);
        \coordinate (v_4) at (-1,1);
        \coordinate (v_5) at (-1,2);
        \coordinate (v_s) at (0,-2);
        \coordinate (v) at (0,1);
        \fill (v_1) circle(3pt) node[anchor=east]{$v_1$};
        \fill (v_2) circle(3pt) node[anchor=west]{$v_2$};
        \fill (v_3) circle(3pt) node[anchor=west]{$v_3$};
        \fill (v_4) circle(3pt) node[anchor=east]{$v_4$};
        \fill (v_5) circle(3pt) node[anchor=east]{$v_5$};
        \fill (v_s) circle(3pt) node[anchor=north]{$v_s$};
        \fill (v) circle(3pt) node[anchor=north]{$v$};
        \draw (v_s) edge[dashed] (v_1);
        \draw (v_4) edge[dashed] (v_5);
        \draw (v_1) -- (v_2) -- (v_3) -- (v_4);
        \draw (v_s) edge[dashed, out=150,in=210] (v_5);
        % \filldraw (.75,0.5) circle(0pt) node[anchor=center]{$C'$};
        \filldraw (.95,0.5) circle(0pt) node[anchor=center]{$C''$};
    \end{tikzpicture}
    \begin{tikzpicture}[baseline=(v_4.base)]
        \coordinate (v_1) at (0,0);
        \coordinate (v_2) at (2,0);
        \coordinate (v_3) at (2,1);
        \coordinate (v_4) at (-1,1);
        \coordinate (v_5) at (-1,2);
        \coordinate (v_s) at (0,-2);
        \coordinate (v) at (0,1);
        \fill (v_1) circle(3pt) node[anchor=east]{$v_1$};
        \fill (v_2) circle(3pt) node[anchor=west]{$v_2$};
        \fill (v_3) circle(3pt) node[anchor=west]{$v_3$};
        \fill (v_4) circle(3pt) node[anchor=east]{$v_4$};
        \fill (v_5) circle(3pt) node[anchor=east]{$v_5$};
        \fill (v_s) circle(3pt) node[anchor=north]{$v_s$};
        \fill (v) circle(3pt) node[anchor=north west]{$v$};
        \draw (v_s) edge[dashed] (v_1);
        \draw (v_4) edge[dashed] (v_5);
        \draw (v) -- (v_1);
        \draw (v_1) -- (v_2) -- (v_3) -- (v_4);
        \draw (v_s) edge[dashed, out=150,in=210] (v_5);
        % \filldraw (.75,0.5) circle(0pt) node[anchor=center]{$C'$};
        \filldraw (1,0.5) circle(0pt) node[anchor=center]{$C''$};
        \filldraw (-.75,0.5) circle(0pt) node[anchor=center]{$C'$};
    \end{tikzpicture}
        }
        \caption{The case $y_4 < y_1.$ Curved, dashed segments represent arbitrary length paths in corresponding cycles. From the left: cycles $C,C_1,C',$ and together, $C_2$ and $C_3,$ respectively.}
        \label{fig:rectangular-cases-2}
    \end{figure}
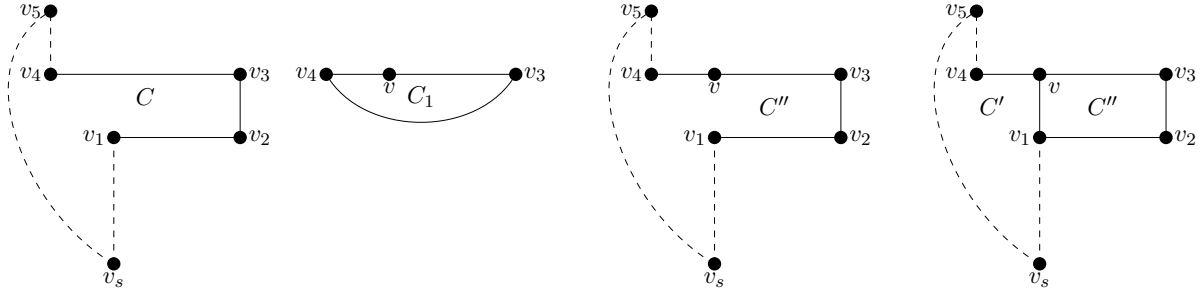

We now deal with the remaining case $y_4 > y_2$, see Figure \ref{fig:rectangular-cases-3}. The vertex $v=(x_4,y_2)$ is not on $C$ since, in that case, $v,v_2,v_3$ would be three vertices on $C$ with the same second coordinate, a contradiction to the hypothesis. Let $C_1=(v_1,v_2,v_3,v),$ and consider
    \begin{gather*}
        \beta = T(C)-T(C_1) \in \kdt{0}{G}.
    \end{gather*}
    Since the edge $\{v_1,v_2\}$ appears with the same coefficient $1$ in both $T(C)$ and $T(C_1),$ it vanishes in $\beta.$ Thus, the support of $\beta$ contains all edges of the cycle $C'=(v,v_3,v_4,\dots,v_s,v_1),$ and, possibly, the edge $\{v_2,v_3\}$. If $G'$ is the subgraph of $G$ consisting of the edges and vertices in the support of $\beta,$ then $G'$ cannot have a vertex of degree $1$ as $\beta\in \kdt{0}{G}$. Hence, the edge $\{v_2,v_3\}$ is not in the support of $\beta$ and the support of $\beta$ is precisely the set of edges of the cycle $C'$.
    
    By Proposition \ref{prop:TC}, there is some nonzero $c$ such that:
    \begin{gather*}
        \beta = T(C) - T(C_1) = c T(C'), \text{ or,}\\
        T(C) = cT(C')+T(C_1).
    \end{gather*}
    Here, $C'$ is a cycle in $C$ of length $s$ and no diagonal edges. Moreover, $C'$ has three vertices with the same first coordinate (in fact, two consecutive horizontal edges). By the first part of the inductive argument and the inductive hypothesis, $T(C') \in \idt{1}{G}.$ Since $C_1$ is a rectangular cycle of length $4$, by Proposition \ref{corollary:rectangular-4-cycle-field}, we have $T(C_1) \in \idt{1}{G}.$ Thus, we conclude that $T(C) = cT(C')+T(C_1) \in \idt{1}{G}$. This finishes our proof.

\begin{figure}[htbp!]
        \centering
        \resizebox{0.75\textwidth}{!}{
                \begin{tikzpicture}[baseline=(v_3.base)]
        \coordinate (v_1) at (0,0);
        \coordinate (v_2) at (2,0);
        \coordinate (v_3) at (2,1);
        \coordinate (v_4) at (3,1);
        \coordinate (v_5) at (3,-2);
        \coordinate (v_s) at (0,-1);
        \fill (v_1) circle(3pt) node[anchor=east]{$v_1$};
        \fill (v_2) circle(3pt) node[anchor=west]{$v_2$};
        \fill (v_3) circle(3pt) node[anchor=south]{$v_3$};
        \fill (v_4) circle(3pt) node[anchor=south]{$v_4$};
        \fill (v_5) circle(3pt) node[anchor=north]{$v_5$};
        \fill (v_s) circle(3pt) node[anchor=east]{$v_s$};
        \draw (v_s) edge[dashed] (v_1);
        \draw (v_4) edge[dashed] (v_5);
        \draw (v_1) -- (v_2) -- (v_3) -- (v_4);
        \draw (v_s) edge[dashed,out=-90,in=210] (v_5);
        \draw (1.5,-1) circle(0pt) node[anchor=center]{$C$};
    \end{tikzpicture}
    \hspace{10mm}
    \begin{tikzpicture}[baseline=(v_3.base)]
        \coordinate (v) at (0,1);
        \coordinate (v_1) at (0,0);
        \coordinate (v_2) at (2,0);
        \coordinate (v_3) at (2,1);
        \fill (v) circle(3pt) node[anchor=east]{$v$};
        \fill (v_1) circle(3pt) node[anchor=east]{$v_1$};
        \fill (v_2) circle(3pt) node[anchor=west]{$v_2$};
        \fill (v_3) circle(3pt) node[anchor=west]{$v_3$};
        \draw (v_1) -- (v_2) -- (v_3) -- (v) -- (v_1);
        \draw (1,0.5) circle(0pt) node[anchor=center]{$C_1$};
    \end{tikzpicture}
    \hspace{10mm}
    \begin{tikzpicture}[baseline=(v_3.base)]
        \coordinate (v) at (0,1);
        \coordinate (v_1) at (0,0);
        \coordinate (v_3) at (2,1);
        \coordinate (v_4) at (3,1);
        \coordinate (v_5) at (3,-2);
        \coordinate (v_s) at (0,-1);
        \fill (v) circle(3pt) node[anchor=east]{$v$};
        \fill (v_1) circle(3pt) node[anchor=east]{$v_1$};
        \fill (v_3) circle(3pt) node[anchor=south]{$v_3$};
        \fill (v_4) circle(3pt) node[anchor=west]{$v_4$};
        \fill (v_5) circle(3pt) node[anchor=west]{$v_5$};
        \fill (v_s) circle(3pt) node[anchor=east]{$v_s$};
        \draw (v_s) edge (v_1);
        \draw (v_4) edge (v_5);
        \draw (v_1) -- (v) -- (v_3) -- (v_4);
        \draw (v_s) edge[out=-90,in=210] (v_5);
        \draw (1.5,-0.5) circle(0pt) node[anchor=center]{$C'$};
    \end{tikzpicture}
        }
        \caption{The case $y_3 < y_4.$ Curved, dashed segments represent arbitrary length paths in corresponding cycles. From the left: cycles $C,C_1,C',$ and together, $C_2$ and $C_3,$ respectively.}
        \label{fig:rectangular-cases-3}
    \end{figure}
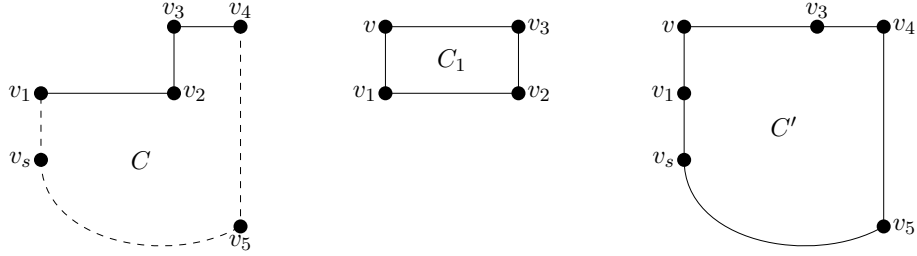
\end{proof}

\subsubsection{Cycles with Diagonal Edges}

As before, assume that $M$ is a Latin square of order $n\geq 4$ and $G$ is the strongly regular graph associated with $M$. Let $\FF$ be an arbitrary but fixed field. 

In this section, we will show that if $C$ is any cycle in $G$, then $T(C) \in \idt{1}{G}$. We use induction on the number of diagonal edges of the cycle and we reduce cycles with diagonal edges to rectangular cycles that we resolved in the previous section. In what follows, $\rho(C)$ denotes the number of diagonal edges of a cycle $C$ in $G$.

\begin{theorem} \label{vanishing-latin-h1}
    Let $M$ be a Latin square of order $n \geq 4,$ and $G$ the strongly regular graph associated with $M$. Let $\FF$ be a field. If $C$ is a cycle in $G$, then $T(C) \in \idt{1}{G}.$ Consequently, $\dim(\idt{1}{G})=\frac{3n^2(n-1)}{2}-n^2+1$ and $\hlcl{G}{\FF}=0.$
\end{theorem}

\begin{proof}
    The proof will be based on $t,$ the number of diagonal edges in $C.$ If $t=0,$ then $C$ is a rectangular cycle, and by Proposition \ref{prop:rectangular-cycles}, we have $T(C) \in \idt{1}{G}.$ Assume that if $C$ is any cycle with $\rho(C) < t,$ then $T(C) \in \idt{1}{G},$ where $t \geq 1$ is fixed. Now, fix a cycle $C$ with $\rho(C)=t.$

    Choose any diagonal edge $e=\{v_1,v_2\}$ on $C$, and let $C=(v_1,v_2,\dots,v_s)$ for some $s \geq 3.$ If $s=3$, then $C$ is a triangle and $T(C)= \delta_1^T(C)$ or $-\delta_1^T(C),$ and in either case, $T(C) \in \idt{1}{G}.$ So, we assume that $s > 3.$ 

    Let $v_1=(x_1,y_1), v_2=(x_2,y_2),$ where necessarily, $x_1 \not = x_2$ and $y_1 \not = y_2.$ We consider two cases as follows.
    \begin{enumerate}[\textbf{Case}(1):]
        \item At least one of the two vertices $(x_1,y_2)$ and $(x_2,y_1)$ is not on $C.$ 
        
        Without loss of generality, assume that $v=(x_1,y_2)$ is not on $C$. Therefore, the edges $\{v_1,v\}$ and $\{v_2,v\}$ are also not in $C$. Moreover, $C_1=(v_1,v_2,v)$ is a triangle in $G$ and 
        \begin{gather*}
            \gamma = T(C) - T(C_1) \in \kdt{0}{G}.
        \end{gather*}
        Since the edge $\{v_1,v_2\}$ appears with coefficient $1$ in both $T(C)$ and $T(C_1),$ it vanishes in $\gamma$. Because $C$ and $C_1$ share exactly one edge $\{v_1,v_2\}$, the remaining edges of $C_1$ and $C$ appear in $\gamma$ with coefficients in $\{1,-1\}$. Thus, if we denote $C'=(v_1,v,v_2,\dots,v_s)$, then, by Proposition \ref{prop:TC}, there is a nonzero constant $c$ such that:
        \begin{gather*}
            \gamma = T(C) - T(C_1) = cT(C'), \text{ or,}\\
            T(C) = T(C_1) + cT(C').
        \end{gather*}
        Because $\{v,v_1\}$ is horizontal and $\{v,v_2\}$ is vertical, $\rho(C') = t-1 < \rho(C)$. By the inductive hypothesis, $cT(C') \in \idt{1}{G}.$ As $C_1$ is a triangle, $T(C_1) \in \idt{1}{G}$. Hence, $T(C)=T(C_1)+cT(C') \in \idt{1}{G}$ and the proof of this case is complete.
        
        \item We now consider the case in which both the vertices $(x_1,y_2)$ and $(x_2,y_1)$ are on $C.$ Denote $v=(x_1,y_2).$ We consider two sub-cases.
        \begin{enumerate}[\textbf{Case}(a)]
            \item The vertex $v$ is in $C$ and at least one of the edges $\{v,v_1\}$ and $\{v,v_2\}$ is in $C$. Since $s>3$, not both of these edges can be in $C$. Without loss of generality, assume that $\{v,v_2\}$ is in $C$ and $\{v,v_1\}$ is not in C. Therefore, $v_1,v_2,v$ are three consecutive vertices on $C$, in that order. %Moreover, since $s > 3,$ the edge $\{v,v_1\}$ is not on $C.$ 
            
            Denote $C_1=(v_1,v_2,v).$ Let
            \begin{gather*}
                \varepsilon = T(C)-T(C_1) \in \kdt{0}{G}.
            \end{gather*}
            By a similar argument as above, the edge $\{v_1,v_2\}$ vanishes in the support of $\varepsilon$. Thus, the support of $\varepsilon$ consists of all the edges of the cycle $C'=(v_1,v,v_4,\dots,v_s)$, and possibly the edge $\{v,v_2\}$. 
            
            But since no edge of the cycle $C'$ is incident with the vertex $v_2$ and $\varepsilon \in \kdt{0}{G},$ we conclude that the edge $\{v,v_2\}$ is not in the support of $\varepsilon$. 
            
            Consequently, the $\kdt{0}{G}$ element $\varepsilon$ has the same support as the edge set of the cycle $C'$. By Proposition \ref{prop:TC}, there is some nonzero $c \in \FF$ such that:
            \begin{gather*}
                \varepsilon = T(C)-T(C_1) = cT(C'), \text{ or,}\\
                T(C) = T(C_1) + cT(C').
            \end{gather*}
            Finally, observe that
            \begin{gather*}
                \text{set of edges of $C'$} = \big(\text{set of edges of $C$} \setminus (\{v_1,v_2\}\cup \{v_2,v\}) \big) \bigcup \{\{v,v_1\}\},
            \end{gather*}
            where the edge $\{v,v_1\}$ is horizontal. Thus, $\rho(C')=t-1 < \rho(C),$ and by the inductive hypothesis, $T(C') \in \idt{1}{G}.$ Since $C_1$ is a triangle, we have $T(C_1) \in \idt{1}{G},$ and thus, $T(C)=T(C_1) + cT(C') \in \idt{1}{G}.$ 
            
            \item In this sub-case, $v$ is on $C$ but none of the edges $\{v,v_1\}$ and $\{v,v_2\}$ is on $C.$ Let $v=v_i$ for some $4 \leq i \leq s-1.$ Now $C_1=\{v_1,v_i,v_2\}$ is a triangle in $G$, which shares exactly one edge with $C$, namely $\{v_1,v_2\}$. We may choose $c\in \{1,-1\}$ such that in the $\kdt{0}{G}$ element $\varphi$ defined by:
            \begin{gather*}
                \varphi = T(C)+cT(C_1),
            \end{gather*}
            the term $\{v_1,v_2\}$ vanishes. Hence, the support of $\varphi$ is precisely the set of edges in the cycles $C_2=(v_2,v_3,\dots,v_i)$ and $C_3=(v_1,v_i,\dots,v_s)$. Note that the set of diagonal edges appearing in $C_2$ or $C_3$ is a proper subset of the set of diagonal edges appearing in $C$. By Lemma \ref{lemma:two-cycle-join-vertex}, there exists $c_2,c_3 \in \FF$ such that:
            \begin{gather*}
                \varphi = c_2T(C_2)+c_3T(C_3).
            \end{gather*}
            However, $\rho(C_2) < \rho(C)=t$ and $\rho(C_3)< \rho(C)=t$. Thus, by the induction hypothesis,  $T(C_2), T(C_3) \in \idt{1}{G}$. Consequently,
            \begin{gather*}
                T(C)=\varphi-cT(C_1)=c_2T(C_2)+c_3T(C_3)-cT(C_1) \in \idt{1}{G}.
            \end{gather*}
        \end{enumerate}
        Thus, in both sub-cases of Case (2), we have $T(C) \in \idt{1}{G}.$
    \end{enumerate}
    Thus, in all cases, $T(C) \in \idt{1}{G}$. By the induction principle, we conclude that $T(C) \in \idt{1}{G}$ for every cycle $C$ in $G$. Consequently,
    \begin{gather*}
        \kdt{0}{G} \subset \idt{1}{G} \subset \kdt{0}{G} \implies \kdt{0}{G} = \idt{1}{G} \implies \hlcl{G}{\FF}=0.
    \end{gather*}
    Furthermore, since $G$ is connected, by Theorem \ref{thm:dim-cycle-space}, we have:
    \begin{gather*}
        \didt{1}{G}=\dkdt{0}{G}=|E(G)|-|V(G)|+1= \frac{3n^2(n-1)}{2}-n^2+1.
    \end{gather*}
    This completes the proof of the theorem.
\end{proof}

\subsection{$H_2$ of Clique Complexes of Latin Square Graphs} \label{section:h-2-latin-squares}

Let $M$ be a Latin square of order $n \geq 4$ and $G$ the strongly regular graph associated with $M.$ Let $\FF$ be an arbitrary field. In the last section, we have proved that $\didt{1,\FF}{G}=\dkdt{0,\FF}{G}$. Denote the strongly regular graph parameters of $G$ by $(v,k,\lambda,\mu)=(n^2,3(n-1),n,6).$ The number of triangles in $G$ is given by 
\begin{gather*}
    \frac{1}{3}|E(G)|\lambda = \frac{3n^2(n-1)}{6} n = \frac{n^3(n-1)}{2}. 
\end{gather*}
By the fundamental theorem of linear algebra,
\begin{gather} \label{eq:eq-3}
    \dkdt{1,\FF}{G} = \frac{n^3(n-1)}{2} - \didt{1,\FF}{G}
    = \frac{n^3(n-1)}{2} - \large\left(\frac{3n^2(n-1)}{2}-n^2+1 \large \right)\\ =\frac{n^2(n^2-4n+5)}{2}-1.
\end{gather}

Recall that an intercalate in the Latin square $M$ is a $2 \times 2$ Latin subsquare of $M$. We will denote by $I(M)$ the number of intercalates of $M$. Observe that if $v_1,v_2,v_3,v_4$ are four distinct vertices forming a $K_4$ in $G$, then exactly one of the following four statements is true:
\begin{itemize}
    \item $v_1,v_2,v_3,v_4$ have the same first coordinates,
    \item $v_1,v_2,v_3,v_4$ have the same second coordinates,
    \item $v_1,v_2,v_3,v_4$ have pairwise distinct first coordinates, pairwise distinct second coordinates, and $M[v_1]=M[v_2]=M[v_3]=M[v_4],$
    \item $v_1,v_2,v_3,v_4$ is an intercalate in $M.$
\end{itemize}
Let $S_1,S_2,S_3,S_4$ be the families of subgraphs of $G$ isomorphic to $K_4$, satisfying the above conditions, respectively. Furthermore, if $G_1$ and $G_2$ are subgraphs of $G$ isomorphic to $K_4$ and $G_1 \in S_i$ and $G_2 \in S_j$ with $i \not = j,$ then $G_1$ and $G_2$ do not share a triangles. Thus, by grouping the columns of $\delta_2^T$ by these families, we see that $\delta_2^T$ may be given the form $A_1 \oplus A_2 \oplus A_3 \oplus A_4,$ where $A_i,i=1,2,3$ can be identified with $\bigoplus_{j=1}^n \delta_{2,\FF}^T(K_n).$ It is well known that $\didt{2,\FF}{K_n}=\binom{n-1}{3}$ (Theorem 8 in \cite{meshulam-newman-rabinovich}).  Consequently,
\begin{gather*}
    \didt{2,\FF}{G} = 3n \didt{2,\FF}{K_n} + \rankg{\FF}{A_4} = 3n\binom{n-1}{3} + \rankg{\FF}{A_4}.
\end{gather*}
Finally, we observe that if $G_1$ and $G_2$ are two distinct subgraphs of $G$ belonging to $S_4,$ then their intersection cannot contain a triangle. Consequently, the columns of $A_4$ have disjoint support, and hence, are linearly independent. Thus, $\rankg{\FF}{A_4}=I(M).$ We have thus established:
\begin{prop} \label{prop:dim-im-d-2t}
    Let $M$ be a Latin square of order $n \geq 4$ and $G$ the strongly regular graph associated with $M$. If $\FF$ is any field, then:
    \begin{gather*}
        \didt{2,\FF}{G} = 3n\binom{n-1}{3} + I(M).
    \end{gather*}
\end{prop}
In particular, $\didt{2,\FF}{G}$ is a function of $n$ and $I(M),$ and is independent of the field $\FF.$
We also have:
\begin{theorem} \label{theorem:latin-square-h2}
    Let $M$ be a Latin square of order $n \geq 4$ and $G$ the strongly regular graph associated with $M$. If $\FF$ is any field, then:
    \begin{gather*}
        \dim(H_2(\clc{G}, \FF)) = (n-1)^3 - I(M).
    \end{gather*}
\end{theorem}
In particular, $\dim(H_2(\clc{G}, \FF))$ is a function of $n$ and $I(M),$ and is independent of the field $\FF.$
\begin{proof}
    By the definition of $H_2$ and \eqref{eq:eq-3}, we get that
    \begin{gather*}
        \dim(H_2(\clc{G}, \FF)) = \dkdt{1,\FF}{G}-\didt{2,\FF}{G}\\
        = \frac{n^2(n^2-4n+5)}{2}-1 - \left( 3 n\binom{n-1}{3} + I(M) \right) = (n-1)^3 - I(M).
    \end{gather*}
\end{proof}

The number of intercalates, $I(M)$, in a Latin square $M$ is a classic invariant that has been studied by several authors. See, for example, \cite{mckay-wanless}, \cite{kwan}, and \cite{kwan-sah}. We recall the following bound on $I(M)$ from Theorem 1 in \cite{katherine-wallis}. We also note that \cite{katherine-wallis} used our symbol $I$ for a different quantity.

\begin{theorem} \label{theorem:interclates-bound}
    Let $M$ be a Latin square of order $n$. Then,
    \begin{enumerate}
        \item if $n$ is even, then $I(M) \leq \frac{n^2(n-1)}{4},$ and,
        \item if $n$ is odd, then $I(M) \leq \frac{n(n-1)(n-3)}{4}.$
    \end{enumerate}
\end{theorem}

Theorem \ref{theorem:latin-square-h2} and Theorem \ref{theorem:interclates-bound} imply the following result.

\begin{theorem} \label{theorem:h-2-latin-square-bounds}
    Let $M$ be a Latin square of order $n \geq 5$ and $G$ the associated strongly regular graph. If $\FF$ is any field, then,
    \begin{enumerate}
        \item if $n$ is even,
        \begin{gather*}
            (n-1)^3 - \frac{n^2(n-1)}{4} \leq \dim(H_2(\clc{G}, \FF)) \leq (n-1)^3,
        \end{gather*}
        \item if $n$ is odd,
        \begin{gather*}
            (n-1)^3 - \frac{n(n-1)(n-3)}{4} \leq \dim(H_2(\clc{G}, \FF)) \leq (n-1)^3.
        \end{gather*}
    \end{enumerate}
    In particular,
    \begin{gather*}
        H_2(\clc{G}, \FF) \not = 0, \text{ and, } \dim(H_2(\clc{G}, \FF) = \Theta(n^3).
    \end{gather*}
\end{theorem}

\begin{proof}
    The proofs of the first two statements follow by applying the bounds in Theorem \ref{theorem:interclates-bound} to the expression for $\dim(H_2)$ described in Theorem \ref{theorem:latin-square-h2}. The last two statements follow by further observing that for $n \geq 5,$
    \begin{gather*}
        0 < (n-1)^3 - \frac{n^2(n-1)}{4} = \frac{3n^3-11n^2+12n-4}{4}, \text{ and,}\\
        0 < (n-1)^3 - \frac{n(n-1)(n-3)}{4} = \frac{3n^3-8n^2+9n-4}{4}.
    \end{gather*}
\end{proof}

\subsection{Higher Homologies of Latin Square Graphs}
\label{subsec:higher-homologies-of-Latin-square-graphs}

Let $M$ be a Latin square of order $n \geq 5$ and $G$ the associated strongly regular graph. Let $\FF$ be an arbitrary field. For $t \geq 5,$ if $S$ is a set of $t$ vertices in $G$ forming a $K_t,$  then $t \leq n$, and exactly one of the following statements is true:
\begin{itemize}
    \item all vertices in $S$ have the same first coordinates,
    \item all vertices in $S$ have the same second coordinates,
    \item the vertices in $S$ have pairwise distinct first coordinates and pairwise distinct second coordinates, and $M[v]$ is the same for all $v \in S.$
\end{itemize}
Thus, for $i=3,\dots,n-2,$ $\delta_{i,\FF}^T(G)$ may be identified as $\bigoplus_{j=1}^{3n} A_j,$ where each $A_j$ is the matrix $\delta_{i,\FF}^T(K_n).$ It is well known that $H_j(\clc{K_n}, \FF)=0$ for each $0 < j.$ (This can also be seen directly by applying a rank-nullity argument to Theorem 8 in \cite{meshulam-newman-rabinovich}.) Thus, we conclude that:

\begin{prop} \label{prop:higher-homology-latin-square}
    Let $M$ be a Latin square of order $n \geq 5$ and $G$ the associated strongly regular graph. Let $\FF$ be an arbitrary field. Then, for $i \geq 3,$ we have:
    \begin{gather*}
        \dim(H_i(\clc{G}, \FF))=0.
    \end{gather*}
\end{prop}

\subsection{Summary of Homology Computations}
In this subsection, we collect our results from the previous subsections and organize them into one theorem.

\begin{theorem} \label{theorem:homology-latin-square}
    Let $M$ be a Latin square of order $n \geq 5$ and $G$ the associated strongly regular graph. Let $\FF$ be an arbitrary field. Denote by $I(M)$ the number of intercalates in $M$. Then, for $i > 0$,
    \begin{gather*}
        \dim(H_i(\clc{G}, \FF)) = \begin{cases}
            (n-1)^3 - I(M), & \text{if } i = 2\\
            0, & i=1 \text{ or } i \geq 3.
        \end{cases}
    \end{gather*}
\end{theorem}

\begin{proof}
     By Theorem \ref{vanishing-latin-h1}, we have $\hlcl{G}{\FF}=0.$ The dimension of $H_2$ is given by Theorem \ref{theorem:latin-square-h2}. Finally, we have $\dim(H_i(\clc{G}, \FF))=0$ for $i \geq 3$ by Proposition \ref{prop:higher-homology-latin-square}. This proves the theorem in all cases.
\end{proof}

\subsection{Spectral and Matrix Theoretic Connections}

In this section, we relate the results from the previous sections with the spectrum of upper and lower Laplacian matrices of the clique complex of a Latin square graph. In the following, for an arbitrary real symmetric matrix $A,$ we denote by $\lambda_{\min}(A)$ the smallest eigenvalue of $A.$ Also, by $L_i(G)$ we will denote the $i^{th}$ total Laplacian of the clique complex of $G$.

\begin{prop}
    Let $M$ be a Latin square of order $n \geq 5$ and $G$ the associated strongly regular graph. Then, all nonzero invariant factors of $\delta_{2, \Z}(\clc{G})$ are $1$'s. That is, if $G$ has $s$ triangles and $t$ copies of $K_4,$ then there are integer matrices $S,T,$ and $D$ such that:
    \begin{itemize}
        \item $T, S, D$ have sizes $t \times t, s \times s,$ and $t \times s,$ respectively,
        \item $\det(S)=\det(T)=1,$
        \item $D$ is a diagonal matrix with all nonzero entries $1$ and all nonzero entries appear consecutively on the diagonal starting at position $(1,1)$, and,
        \item $\delta_{2, \Z}(\clc{G}) = TDS.$
    \end{itemize}
\end{prop}

\begin{proof}
    By Proposition \ref{prop:dim-im-d-2t}, we have that over any field $\FF,$
    \begin{gather*}
        \didt{2,\FF}{\clc{G}} = 3n\binom{n-1}{3} + I(M).
    \end{gather*}
    Therefore, 
    \begin{gather*}
        \rankg{\FF}{\delta_{2,\FF}(\clc{G})} = \rankg{\FF}{\delta_{2,\FF}^T(\clc{G})} = \didt{2,\FF}{(\clc{G})} = 3n\binom{n-1}{3} + I(M).
    \end{gather*}
    In particular, if the $\re$-rank of $\delta_{2,\Z}(\clc{G})$ is $l$ and $p$ is any prime, then,
    \begin{gather*}
        \rankg{\FF_p}{\delta_{2,\FF_p}^T(\clc{G})} = l.
    \end{gather*}
    Since $\delta_{2,\Z}(\clc{G})$ is an integer matrix, we may consider its Smith normal form which yields integer matrices $T,D,S$ of the sizes mentioned above such that $\det(S)=\det(T)=1,$ $\delta_{2,\Z}(\clc{G})=TDS,$ and the nonzero invariant factors $d_1 \mid d_2 \mid \dots \mid d_l$ (where $l$ is the real rank of $\delta_{2,\Z}(\clc{G})$) appearing consecutively at the top-left corner of the upper diagonal matrix $D.$ To complete the proof of the proposition, it suffices to show that all nonzero entries in $D$ are $1$'s. Assume on the contrary that $D$ has a nonzero diagonal entry $d$ which is divisible by some prime $p$. Now, $\det(S)=\det(T)=1,$ which is invertible over $\FF_p.$ Consequently, both $S$ and $T$ are invertible over $\FF_p.$ Thus,
    \begin{gather*}
        \rankg{\FF_p}{\delta_{2,\FF_p}(\clc{G})}  = \rankg{\FF_p}{TDS} = \rankg{\FF_p}{D} < l,
    \end{gather*}
    which is a contradiction. Thus, no nonzero invariant factors of $\delta_{2,\Z}(\clc{G})$ has any prime divisor, i.e., all nonzero invariant factors of $\delta_{2,\Z}(\clc{G})$ must be $1'$s.
\end{proof}

If $G$ is a graph, we define the $i^{th}$ total Laplacian $L_i$ of $G$ as $L_i=\delta_{i-1,G,\re}\delta_{i-1,G,\re}^T+\delta_{i,G,\re}^T\delta_{i,G,\re}.$

\begin{theorem}
    Let $M$ be a Latin square of order $n \geq 5$ and $G$ the associated strongly regular graph. Then,
    \begin{enumerate}
        \item $\lambda_{\min}(L_i) > 0$ if $i=1$ or $i \geq 3,$ and,
        \item $\lambda_{\min}(L_2) = 0.$
    \end{enumerate}
\end{theorem}
\begin{proof}
    For $i \geq 1,$ $L_i$ is positive semidefinite. It is known (see \cite[Theorem 2.2]{HORAK} or \cite[Section 4.3]{Lim}), that, for any $i$:
    \begin{gather*}
        H_i \cong H^i = 0 \iff \ker(L_i) = 0 \iff \lambda_{\min}(L_i) > 0.
    \end{gather*}
    By Theorem \ref{theorem:homology-latin-square}, if $i=1$ or $i \geq 3$, then, $H_i(\clc{G},\re)=0,$ and so, $\lambda_{\min}(L_i) > 0.$ On the other hand, by Theorem, \ref{theorem:h-2-latin-square-bounds}, $H_2(\clc{G}, \re) \not = 0,$ and thus, $\lambda_{\min}(L_2)=0.$
\end{proof}
We also observe that the number of intercalates in $M$ is determined by the spectra of $L_2^{\uparrow}.$

\begin{theorem}
Let $M$ be a Latin square of order $n \geq 5$. If $G$ is the associated strongly regular graph, then the spectrum of $L_2^{\uparrow}(\clc{G})$ is:
    \begin{gather*}
        \sigma(L_2^{\uparrow}(\clc{G})) = \begin{pmatrix}
        0 & 4 & n \\
        3n \binom{n-1}{2} + 3I(M) & I(M) & 3n \binom{n-1}{3}
    \end{pmatrix},
    \end{gather*}
where the first row contains the distinct eigenvalues and the second row lists their respective multiplicities.
    Consequently,
    \begin{gather*}
        \tr{L_2^{\uparrow}(\clc{G})} = 3n^2 \binom{n-1}{3} + 4 I(M),
    \end{gather*}
    and, the number of intercalates in $M$ is given by:
    \begin{gather*}
        I(M) = \frac{1}{4} \left(\tr{L_2^{\uparrow}(\clc{G})}- 3n^2 \binom{n-1}{3}\right). 
    \end{gather*}
\end{theorem}

\begin{proof}
    All quantities below are over the field $\re$. As discussed in Section \ref{section:h-2-latin-squares}, $\delta_2(\clc{G})$ has the following block form:
    \begin{gather*}
        \delta_2(\clc{G}) \cong \left(\bigoplus_{j=1}^{3n} A_j\right) \oplus B,
    \end{gather*}
    where each $A_j$ may be identified with $\delta_2(\clc{K_n}),$ and $B$ is the $I(M) \times 4 I(M)$ matrix, with its rows and columns indexed by the intercalates in $M$, and the triangles in $G$ whose vertices are contained in these intercalates, respectively. Thus, the block diagonal matrix $L_2^{\uparrow} = \delta_2(\clc{G})^T \delta_2(\clc{G})$ has the form:
    \begin{gather*}
        L_2^{\uparrow}(\clc{G}) = \delta_2(\clc{G})^T \delta_2(\clc{G}) = \left( \bigoplus_{j=1}^{3n} A_j^TA_j \right) \oplus B^TB = \left( \bigoplus_{j=1}^{3n} L_2^{\uparrow}(\clc{K_n}) \right) \oplus B^TB.
    \end{gather*}
    Furthermore, since the copies of $K_4$ corresponding to the intercalates in $M$ share no triangles, we have:
    \begin{gather*}
        B = \bigoplus_{j=1}^{I(M)} \delta_2(\clc{K_4}),
    \end{gather*}
    and thus,
    \begin{gather*}
        B^TB = \bigoplus_{j=1}^{I(M)} L_2^{\uparrow}(\clc{K_4}).
    \end{gather*}
    From \cite[Example 2.3]{BGP}, we know that, for any $m$,
    \begin{gather*}
        \sigma(L_2^{\uparrow}(\clc{K_m})) = \begin{pmatrix}
            0 & m\\
            \binom{m-1}{2} & \binom{m-1}{3}.
        \end{pmatrix}
    \end{gather*}
    We obtain that
    \begin{gather*}
        \sigma(L_2^{\uparrow}(\clc{G})) = \sigma \left( \bigoplus_{j=1}^{3n} L_2^{\uparrow}(\clc{K_n}) \right) \bigcup \sigma(B^TB) = \sigma \left( \bigoplus_{j=1}^{3n} L_2^{\uparrow}(\clc{K_n}) \right) \bigcup \sigma\left(\bigoplus_{j=1}^{I(M)} L_2^{\uparrow}(\clc{K_4})\right)\\
        = \begin{pmatrix}
            0 & n \\
            3n \binom{n-1}{2} & 3n \binom{n-1}{3}
        \end{pmatrix}
        \bigcup
        \begin{pmatrix}
            0 & 4 \\
            3 I(M) & I(M)
        \end{pmatrix}
        = \begin{pmatrix}
            0 & 4 & n \\
            3n \binom{n-1}{2} + 3I(M) & I(M) & 3n \binom{n-1}{3}
        \end{pmatrix}.
    \end{gather*}
    This proves the first statement. We thus obtain:
    \begin{gather*}
        \tr{L_2^{\uparrow}(\clc{G})} = 3n^2 \binom{n-1}{3} + \tr{B^TB} = 3n^2 \binom{n-1}{3} + 4 I(M),
    \end{gather*}
    and so,
    \begin{gather*}
        I(M) = \frac{1}{4} \left(\tr{L_2^{\uparrow}(\clc{G})}- 3n^2 \binom{n-1}{3}\right).
    \end{gather*}
    This completes the proof.
\end{proof}
We also observe that for $i \geq 3$, the spectrum of $L_i^{\uparrow}$ is determined by $n$ and $I(M)$, and contains no further information regarding $M$ or $G.$
\begin{prop}
    Let $M$ be a Latin square of order $n \geq 5$ and $G$ the associated strongly regular graph. Then,
    \begin{gather*}
        \sigma(L_3^{\uparrow}) = \begin{pmatrix}
            0 & n \\
            3n \binom{n-1}{3}+I(M) & 3n \binom{n-1}{4}
        \end{pmatrix}.
    \end{gather*}
    For $4 \leq i \leq n-2,$ we have:
    \begin{gather*}
        \sigma(L_i^{\uparrow}) = \begin{pmatrix}
            0 & n \\
            3n \binom{n-1}{i} & 3n \binom{n-1}{i+1}
        \end{pmatrix}.
    \end{gather*}
\end{prop}
\begin{proof}
    If $S$ is a clique in $G$ such that $|S| \geq 5,$ then exactly one of the following three statements is true:
    \begin{itemize}
        \item all vertices in $S$ have the same first coordinate,
        \item all vertices in $S$ have the same second coordinate,
        \item $M$ has the same symbol at every vertex belonging to $S.$
    \end{itemize}
    Therefore, if $i>3$ the number of $i$-simplices in $\clc{G}$ is $t,$ then $L_i^{\uparrow}(\clc{G})$ may be identified with 
    \begin{equation*}
    \bigoplus_{j=1}^{3n} L_i^{\uparrow}(\clc{K_n}).
    \end{equation*}
    
    Since each $L_i^{\uparrow}(\clc{K_n})$ has spectrum 
    $\begin{pmatrix}
            0 & n\\
            \binom{n-1}{i} & \binom{n-1}{i+1},
    \end{pmatrix}$
    by \cite{BGP}, the proposition follows for $i>3$. If $i=3,$ $L_i^{\uparrow}(\clc{G})$ contains an additional $I(M)$ zero rows and columns, which contributes $I(M)$ to the multiplicity of the $0$-eigenvalue, and the proposition follows.
\end{proof}

\begin{rk}
We note that if $G_1$ and $G_2$ are the strongly regular graphs associated with two Latin squares of the same order with the same number of intercalates, then $G_1$ and $G_2$ cannot be distinguished by their corresponding $\bigcup_{0 \leq i \leq n, i \not = 1} \sigma(L_i^{\uparrow}).$
\end{rk}

\section{\texorpdfstring{Graphs Associated with Orthogonal Arrays $\text{OA}(t,n)$}{}}
\label{sec:graphs-associated-with-orthogonal-arrays}

Let $t$ and $n$ be two natural numbers. An orthogonal array $\text{OA}(t,n)$ is a $t\times n^2$ matrix with entries from $[n]$ such that the $n^2$ ordered pairs defined by any two rows of the matrix are all distinct. If $L_1,\dots,L_t$ are Latin squares of the same order $n$, they are said to be mutually orthogonal if, for every $i \ne j,$ the $n^2$ pairs of values $(L_i[x,y],L_j[x,y])$, $(x,y) \in [n]\times [n]$ are all distinct. An orthogonal array $\oa{t}{n}, t \geq 3$ is equivalent to $t-2$ mutually orthogonal Latin squares of order $n$, see \cite[Thm. 10.4.1]{GR}. The graph associated to an orthogonal array $\text{OA}(t,n)$ has as vertices the $n^2$ columns of the array and two columns are adjacent if they agree in exactly one position. It is known that this graph is a strongly regular graph with parameters $(n^2, t(n-1), n-2+(t-1)(t-2),t(t-1))$, see \cite[Thm. 10.4.2]{GR}.

The main theorem of this section is the following.

\begin{theorem} \label{theorem:vanishing-h1-orthogonal-array}
Let $\FF$ be a field and $M$ be an $t \times n^2$ orthogonal array, where $t \geq 3$ and $n \geq 4$. If $G$ is the strongly regular graph associated with $M$,  then $\hlcl{G}{\FF}=0.$
\end{theorem}

\begin{proof} 
    If $t=3,$ then $M$ is equivalent to a Latin square and we have proved in the last section that $\hlcl{G}{\FF}=0$ in this case. 
    
    Thus, we may assume $t \geq 4$. Let $M_1,\dots,M_{t-2}$ be mutually orthogonal Latin squares of order $n$ such that the collection $M_1,\dots,M_{t-2}$ is equivalent to $M$. Equivalently, $M_j$ corresponds to the row $j+2$ of the orthogonal array $M$ for $1\leq j\leq t-2$.
    
    Let $G_1,\dots,G_{t-2}$ be the strongly regular graphs corresponding to $M_1,\dots,M_{t-2},$ respectively. 
    
    For each $i\in [t-2]$, we view each $G_i$ as a subgraph of $G$, and in this identification, observe that any triangle or cycle in some $G_i$ is also a triangle or cycle in $G$.

    Now, let $C$ be an arbitrary cycle in $G$ and consider $T_{\FF}(C).$ It suffices to show that $T_{\FF}(C) \in \idt{1,\FF}{G}.$ If $C$ has no diagonal edges, then, in fact, $C$ is a cycle in $G_1$, and by Theorem \ref{vanishing-latin-h1}, there are triangles $C_1,\dots,C_s$ in $G_1$ and constants $c_1,\dots,c_s$ such that
    \begin{gather*}
        T_{\FF}(C) = \sum_{i=1}^s c_i\delta_{1,\FF}^T(C_i), 
    \end{gather*}
    where $T_{\FF}(C) \in \FF^{E(G_1)}$ and $\delta_{1,\FF}^T$ is the boundary operator defined on $\clc{G_1}.$ On the other hand, since each triangle $C_1,\dots,C_s$ is also in $G,$ we conclude that the same equality
    \begin{gather*}
        T_{\FF}(C) = \sum_{i=1}^s c_i \delta_{1,\FF}^T(C_i)
    \end{gather*}
    holds when we view $T_{\FF}(C) \in \FF^{E(G)}$ and $\delta_{1,\FF}^T$ denotes the corresponding boundary operator on  $\clc{G}.$ Thus, if $C$ is any cycle in $G$ with no diagonal edges, then $T_{\FF}(C) \in \idt{1,\FF}{G}.$

    Now, assume that $C$ has at least one diagonal edge. Let $D=\{e_1,\dots,e_m\}$ be the diagonal edges in $C$. For each such edge $e_i$, there is a unique $1 \leq a_i \leq t-2$ such that some triangle $C_i$ in $G_{a_i}$ contains $e_i$ as its only diagonal edge. That is, the two other edges of $C_i$ are horizontal or vertical, and consequently, both belong to $G_1$ in particular. Thus, we may choose constants $c_1,\dots,c_m \in \{1,-1\}$ such that:
    \begin{gather*}
        \al := \sum_{i=1}^m c_i\delta_{1,\FF}^T(C_i) \in \idt{1,\FF}{G} \subset \kdt{0,\FF}{G},
    \end{gather*}
    and
    \begin{gather*}
        \gamma := T_{\FF}(C) - \al \in \kdt{0,\FF}{G},
    \end{gather*}
    has support consisting of only horizontal and vertical edges. Since all vertical and horizontal edges in $G$ are also contained in the subgraph $G_1$, we may view $\gamma$ as an element of $\kdt{0,\FF}{G_1}.$ Consequently, by Theorem \ref{vanishing-latin-h1}, $\gamma \in \idt{1,\FF}{G_1}$ (viewed in $\FF^{E(G_1)}$), and consequently, $\gamma \in \idt{1,\FF}{G},$ and so, $T_{\FF}(C) = \gamma + \al \in \idt{1,\FF}{G}.$ This completes the proof.
\end{proof}

\section{Block Graphs of Steiner Systems}
\label{sec:block-graphs-of-steiner-systems}

A Steiner system $S(2,m,n)$ is an incidence structure on a set of $n$ points $P$ and a family of $m$-subsets $\mathcal{B}$ of $P$ such that, for each pair of distinct points in $P$, there is exactly one block $B \in \mathcal{B}$ that contains them. The block graph $G$ of a Steiner system $S(2,m,n)$ is the graph whose vertex set is the set of blocks $\mathcal{B}$ and two vertices are adjacent if and only if the corresponding blocks intersect. Note that by the definition of a Steiner system, two adjacent blocks in $G$ have intersection size exactly one. It is known that if $G$ is the block graph of a Steiner system $S(2,m,n),$ then $G$ is a strongly regular graph with parameters
\begin{gather*}
    \left(\frac{n(n-1)}{m(m-1)},\frac{m(n-m)}{m-1}, (m-1)^2+\frac{(n-1)}{m-1}-2, m^2\right).
\end{gather*}

For the rest of the section, if $G$ is the strongly regular graph that is the block graph of a Steiner system $S$, then, for $v \in V(G),$ we will denote by $\phi(v, S)$, or $\phi(v),$ when $S$ is understood, to denote the block of $S$ corresponding to $v.$

\begin{prop} \label{prop:steiner-4-cycle-common-neighbor}
    Let $G$ be the block graph of a Steiner system $S(2,m,n),$ where $m \geq 3.$ If $C=(a,b,c,d)$ is an induced cycle of length $4$ in $G,$ then there is a common neighbor $v$ of $a,b,c,d$ in $G.$
    % Consequently, we may choose coefficients $c_1,c_2,c_3,c_4 \in \{-1,1\}$ such that:
    % \begin{gather*}
    %     T_{\Z}(C) = \sum_{i=1}^4 c_i \delta_{1,\Z}^T(C_i),
    % \end{gather*}
    % where
    % \begin{gather*}
    %     C_1 = \{v,a,b\}, C_2 = \{v,b,c\}, C_3 = \{v,c,d\}, C_4 = \{v,d,a\}.
    % \end{gather*}
    % Thus, if $C'$ is any cycle of length $4$ in $G$, then $T_{\Z}(C) \in \idt{1,\Z}{G}.$
\end{prop}
\begin{proof}
    % It suffices to prove the first statement, since then, the second statement follows from Lemma 5.9 in \cite{cioaba-guo-ji-mim}.
    % If $C$ 
    % If $C$ is induced, the second statement applies, and if $C$ is not induced, the third statement follows by an application of Lemma \ref{lemma:cycle-split-over-chord}.
    Since $a \sim b$ and $a \sim d$, but $b \not \sim d$, without loss of generality and up to relabeling the points in $S$, $\phi(a), \phi(b), \phi(d)$ must have the following forms: 
    \begin{gather*}
        \phi(a) = \{a_1,a_2,\dots,a_m\},\\
        \phi(b) = \{a_1,b_2,\dots,b_m\},\\
        % \phi(d) = \{a_2,c_2,\dots,c_m\},
        \phi(d) = \{a_2,d_2,\dots,d_m\},
    \end{gather*}
    where distinct symbols above represent distinct points in the Steiner system. That is, $a_i \not = b_j, a_i \not = d_k, b_j \not = d_k$ for every $i,j,k.$ Note that here, we have used the observation that:
    \begin{gather*}
        \phi(a) \cap \phi(b) \not = \phi(b) \cap \phi(c), \text{ since } \phi(a) \cap \phi(c) = \emptyset.
    \end{gather*}
    Finally, since $\phi(c) \cap \phi(a) = \emptyset$ and $\phi(c)$ has nonempty intersection with both $\phi(b)$ and $\phi(d),$ up to relabeling of the points in $S,$ $\phi(c)$ has the form:
    \begin{gather*}
        \phi(c) = \{b_2,d_2,c_3,\dots,c_m\},
    \end{gather*}
    where $c_i$ is distinct from $a_j,b_j,d_j$ for every $i,j.$ By the definition of the Steiner system $S$, there is exactly one vertex $v \not \in \{a,b,c,d\}$ such that $\phi(v) \supset \{a_1,d_2\}.$ Thus, 
    \begin{gather*}
        \phi(v) \cap \phi(u) \not = \emptyset,\ u \in \{a,b,c,d\}.
    \end{gather*}
    That is, $v$ is a common neighbor of $a,b,c,d,$ as needed.
\end{proof}

\begin{prop} \label{prop:steiner-5-cycle-common-neighbor}
    Let $G$ be the block graph of a Steiner system $S(2,m,n),$ where $m \geq 3.$ If $C$ is an induced cycle of length $5$ in $G,$ then, every four (necessarily consecutive) vertices in $C$ have a common neighbor in $G$.
\end{prop}

\begin{proof}
    Let $C=(a,b,c,d,e).$ We will show that $e,a,b,c$ have a common neighbor in $G.$ By a similar argument as above, up to relabeling of the points of $S$, we may assume that $\phi(a), \phi(b), \phi(e)$ have the following forms:
    \begin{gather*}
        \phi(a) = \{a_1,a_2,\dots,a_m\},\\
        \phi(b) = \{a_1,b_2,\dots,b_m\},\\
        \phi(e) = \{a_2,e_2,\dots,e_m\},
    \end{gather*}
    where distinct symbols refer to distinct points in $S$. Furthermore, since $c \sim b$ and $c \not \sim a, c \not \sim e,$ $\phi(c)$ has the following form, up to relabeling of the points in $S:$
    \begin{gather*}
        \phi(c) = \{b_2,c_2,\dots,c_m\},
    \end{gather*}
    where the $c_i \not = a_j,b_j,e_j$ for every $i,j$. Now, by the definition of the Steiner system $S$, there is a vertex $v$ in $G$ (which is necessarily distinct from $d$, since $C$ is induced), such that $\phi(v) \supset \{a_2,b_2\}.$ That is,
    \begin{gather*}
        \phi(v) \cap \phi(u) \not = \emptyset, \ u \in \{e,a,b,c\},
    \end{gather*}
    which implies $v \sim u$ for $u \in \{e,a,b,c\}.$
\end{proof}

\begin{theorem} \label{thm:vanishing-h1-steiner-system}
    Let $G$ be the block graph of a Steiner system $S(2,m,n)$, where $m \geq 3$. Let $\FF$ be a field. Then, $\kdt{0,\FF}{G}=\idt{1,\FF}{G},$ equivalently $\hlcl{G}{\FF}=0.$
\end{theorem}

\begin{proof}
    By Proposition \ref{prop:steiner-4-cycle-common-neighbor} and Proposition \ref{prop:steiner-5-cycle-common-neighbor}, $G$ satisfies the hypothesis in Theorem \ref{thm:4-vertex-common-neighbor-srg}. The theorem follows.
\end{proof}

\section{Conference Graphs} \label{sec:conference-graphs}

In this section, we prove that all sufficiently large conference graphs have vanishing clique homology over every field. A conference graph $G$ is a strongly regular graph with parameters $(v,\frac{v-1}{2}, \frac{v-5}{4}, \frac{v-1}{4})$ for some positive integer $v \geq 5.$ The nontrivial eigenvalues of this graph are:
\begin{gather*}
    \frac{-1+\sqrt{v}}{2},\ \frac{-1-\sqrt{v}}{2}.
\end{gather*}

%In particular, the spectral gap, $\rho$, of $G$ is $\frac{1+\sqrt{v}}{2}.$ 

Recall that for any connected $k$-regular graph $G$, we define:
\begin{gather*}
    \rho(G) = \max \{|\theta| : \theta \text{ is an eigenvalue of } G \text{ and } \theta \not = k\}. 
\end{gather*}

We will need the following result which is sometimes referred to as the expander mixing lemma, see \cite{AlonChung88} and \cite[Lemma 2.5]{HLW06}.

\begin{lemma}[Lemma 2.5 \cite{HLW06}]\label{lem:expmixlemma}
Let $G$ be a $k$-regular graph with $v$ vertices. If $S$ and $T$ are subsets of the vertex set of $G$, and $e(S,T)$ denotes the number of ordered pairs $(s,t)$ of vertices of $G$ such that $s \in S,\ t \in T$ and $\{s,t\} \in E(G)$, then
\begin{gather*}
        \left|e(S,T) -\frac{k}{v}|S||T|\right| \leq \rho(G) \sqrt{|S||T|}.
    \end{gather*}
\end{lemma}

\begin{defn}
Let $G$ be a graph. For two distinct vertices  $x$ and $y$, we denote by $G_{x,y}:=G[N_G(x) \cap N_G(y)]$ the subgraph of $G$ induced by the common neighborhood of $x$ and $y$.
\end{defn}

\begin{prop} \label{prop:common-neighborhood-two-components}
There exists a universal constant $C_0'$ such that, if $G$ is a conference graph on $v \geq C_0'$ vertices and $x,y$ are two nonadjacent and distinct vertices of $G$ such that $G_{x,y}$ is disconnected, then $G_{x,y}$ consists of exactly two connected components: one which is a single vertex graph and another that consists of all remaining vertices in $G_{x,y}$.   

\end{prop}
\begin{proof}
We first prove that there exists a universal constant $C_0'$ such that: if $G$ is a conference graph on $v \geq C_0'$ vertices, and $x$ and $y$ are any nonadjacent, distinct vertices of $G$, and $S$ is a nonempty subset of $N_G(x) \cap N_G(y)$ such that $|S| \leq \frac{1}{2}(|N_G(x) \cap N_G(y)|)=\frac{\mu}{2}$ and $|e(S, \ (N_G(x) \cap N_G(y))\setminus S)|=0,$ then $|S| \leq 4$.

Let $f: (1, \infty) \to \re$ defined by:
\begin{gather*}
        f(v) = \frac{v^2(1+\sqrt{v})^2}{(v-1)^2}-5\left(\frac{v-1}{4}-5\right) = -\frac{v^3-8v^2\sqrt{v}-14v^2+5v}{4(v-1)^2}+\frac{105}{4}.
\end{gather*}
From the second expression of $f(v)$ above, it follows that $\lim_{x \to \infty}f(x) = -\infty$. 

Let $C_0'$ be the smallest positive integer such that $C_0' > 65$ and $f(x)<0$, for any $x\geq C_0'$.

Assume that $G$ is a conference graph with parameters $(v,k,\lambda,\mu)=(v,\frac{v-1}{2}, \frac{v-5}{4}, \frac{v-1}{4})$, where $v \geq C_0'$. 
Therefore, $\mu \geq 10$. Furthermore, assume that $G$ has two distinct, nonadjacent vertices $x$ and $y$ such that $G_{x,y}$ is disconnected. 

Let $S$ be a nonempty subset of $N_G(x) \cap N_G(y)$ such that $|S| \leq \frac{\mu}{2}$ and $|e(S,\  (N_G(x) \cap N_G(y)\setminus S)|=0.$ We will show that $|S| \leq 4$. 

Assume on the contrary that $s:=|S| \geq 5$. We have that $|S|+|T|=|N_G(x)\cap N_G(y)|=\mu=\frac{v-1}{4}>10$ and $|S||T|= s(\mu-s)$. The function $g:\re \to \re$ defined by $g(x)=x(\mu-x)$ is increasing on $(-\infty, \frac{\mu}{2}).$ By the hypothesis, $5 \leq s \leq \frac{\mu}{2},$ and we deduce that $5(\mu-5) \leq |S||T|$. Applying Lemma \ref{lem:expmixlemma} to $S$ and $T$, we get that
    \begin{gather*}
        \left| e(S,T) -\frac{k}{v}|S||T|\right| \leq \rho(G) \sqrt{|S||T|} 
    \end{gather*}
    Simplifying after using the hypothesis $|e(S,T)|=0$ this yields:
    \begin{gather*}
        |S||T| \leq \frac{\rho(G)^2v^2}{k^2}.
    \end{gather*}
    Thus,
    \begin{gather*}
        5\left(\frac{v-1}{4}-5\right) = 5(\mu-5) \leq |S||T| \leq \frac{\rho(G)^2v^2}{k^2} = \frac{\left( \frac{1+\sqrt{v}}{2}\right)^2v^2}{\left(\frac{v-1}{2}\right)^2}=\frac{v^2(1+\sqrt{v})^2}{(v-1)^2}.
    \end{gather*}
    Consequently,
    \begin{gather*}
        f(v) = \frac{v^2(1+\sqrt{v})^2}{(v-1)^2} - 5\left(\frac{v-1}{4}-5\right) \geq 0.
    \end{gather*}
    However, this contradicts that assumption that $f(v) < 0.$ We conclude that $|S| \geq 5$ is not possible, and our claim follows. 

We now prove that if $S,T$ are sets of vertices of $G_{x,y}$ such that $0 < |S| \leq \frac{|V(G_{x,y})|}{2}=\frac{\mu}{2},$ $S \cap T = \emptyset,$ $S \cup T = V(G_{x,y}),$ and $e(S,T)=0,$ then $|S|=1.$ Assume on the contrary that $S$ has at least two distinct vertices $a,b$. By our previous claim, we have $|S|\leq 4.$\par 

Let $W=N_G(x) \cap N_G(y)$, and $X = N_G(x) \setminus W$. Hence, $|X|=|N_G(x)|-|W|=\frac{v-1}{2}-\frac{v-1}{4}=\mu.$ Thus,
    \begin{gather*}
        N_G(x) = S \sqcup T \sqcup X, \text{ (where $\sqcup$ means disjoint union)}.    
    \end{gather*}
    Since $a \sim x$, we have $|N_G(a) \cap N_G(x)|=\lambda=\mu-1$. 
    
    Because $a \in S$ and $e(S,T)=0$, we have that $|N(a)\cap T|=0$. Also, as $|S|\leq 4$ and $a\in S$, $|N(a)\cap S|\leq 3$. Hence, $|N(a)\cap W|\leq 3$ as $W=S\sqcup T$.
    
    Therefore, $\mu-1=|N(a) \cap N(x)|=|N(a) \cap X| +|N(a)\cap W|\leq |N(a)\cap X|+3$ and so, $|N(a) \cap X|\geq \mu-4$.   
    
    Since $b \sim x$ and $b \in S,$ a similar argument shows that $|N(b) \cap X|\geq \mu-4$. By the inclusion-exclusion principle,
    \begin{align*}
        |(N(a) \cap X)\cap (N(b) \cap X)|& = |N(a) \cap X| + |N(b) \cap X| - |(N(a) \cap X) \cup (N(b) \cap X)|\\ 
        &\geq |N(a) \cap X| + |N(b) \cap X| - |X|\\\    
        &\geq 2(\mu-4)-\mu=\mu-8.
    \end{align*}
    Hence, $a$ and $b$ have at least $\mu-8$ common neighbors in $X = N(x) \setminus N(y)$. As $a$ and $b$ are both adjacent to $y$, we may argue similarly and conclude that $a$ and $b$ have at least $\mu-8$ common neighbors in $N(y) \setminus N(x)$. Thus, the number of common neighbors of $a$ and $b$ is at least $2(\mu-8)=2\mu-16$.

    However, since $v > 65$, we have $\mu=\frac{v-1}{4} > 16,$ and so, $2\mu-16 > \mu= \max\{\mu, \lambda\}.$ Thus, the distinct vertices $a,b$ cannot have $2\mu-16$ or more common neighbors. This gives us a contradiction, and we conclude that $|S|$ cannot have more than one vertex. This completes the proof of the proposition.    
\end{proof}

\begin{rk}
We claim that $C_0'=256$ suffices: that is, for $v \geq 256$,
$$
f(v) = -\frac{v^{3} - 8 \, v^{\frac{5}{2}} - 119 \, v^{2} + 215 \, v - 105}{4 \, {\left(v-1\right)^2}} < 0
$$
holds. Consequently, it suffices to prove that for $v \geq 256$,
$$
v^{3} - 8 \, v^{\frac{5}{2}} - 119 \, v^{2} + 215 \, v - 105 > 0
$$
 holds. Fix $v \geq 256$. Observe that $215v-105>0$. Furthermore, $8v^{\frac{1}{2}}\geq 128$, and so, $8v^{\frac{5}{2}} \geq 128v^2 > 119v^2$. Thus,
$$
v^{3} - 8 \, v^{\frac{5}{2}}-119v^2 > v^3-8v^{\frac{5}{2}}-8v^{\frac{5}{2}}=v^3-16v^{\frac{5}{2}}.
$$
Also, $v \geq 256$ gives us $v \geq 16v^{\frac{1}{2}},$ and, $v^3 \geq 16v^{\frac{5}{2}}$. Combining the last two inequalities, we have:
$$
v^{3} - 8 \, v^{\frac{5}{2}}-119v^2 > v^3-16v^{\frac{5}{2}} \geq 0.
$$
Thus,
$$
v^{3} - 8 \, v^{\frac{5}{2}} - 119 \, v^{2} + 215 \, v - 105 = (v^{3} - 8 \, v^{\frac{5}{2}} - 119 \, v^{2})+(215v-105) > 0 + 0= 0.
$$
This completes the proof of the claim.
\end{rk}

\begin{theorem} \label{thm:conference-induced-4-cycle-across-vertices}
Let $G$ be a conference graph on $v \geq C_0'$ vertices, and let $(a,b,c,d)$ be an induced cycle in $G$. If the vertices $b$ and $d$ are in different components of $G_{a,c},$ then $a$ and $c$ are in the same component of $G_{b,d}.$
\end{theorem}

\begin{proof}
    We may assume $G_{a,c}$ is disconnected, since otherwise the theorem is immediate. The hypothesis implies $G_{a,c}$ has exactly two connected components, one of which is an isolated vertex. Furthermore, one of the vertices $b$ and $d$ must form the single vertex component. Without loss of generality, we may assume that $b$ is the single vertex component. Let the strongly regular graph parameters of $G$ be $(v,k,\lambda,\mu)=(v, \frac{v-1}{2}, \frac{v-5}{4}, \frac{v-1}{4})$.  Denote:
    \begin{gather*}
        W = N_G(a) \cap N_G(c), \quad A = N_G(a) \setminus W, \quad C = N_G(c) \setminus W.
    \end{gather*}
    That is,
    \begin{gather*}
        N_G(a) = A \sqcup W, \quad N_G(c) = C \sqcup W.
    \end{gather*}
    Here, $|A|=|C|=k-\mu=\mu$.\par 
    
    Since $b \sim a,$ we have $|N_G(a) \cap N_G(b)|=\lambda=\mu-1.$ Because $b$ is an isolated vertex in the subgraph induced by $W$, $|N_G(b) \cap W|=0.$ Hence, $|N_G(b)\cap A|=\mu-1=|A|-1.$ This means that there is a unique $a_0 \in A$ such that $N_G(b) \cap N_G(a) = A \setminus \{a_0\}.$ Similarly, since $b \sim c,$ we conclude that there is a unique $c_0 \in C$ such that $N_G(b) \cap N_G(c) = C \setminus \{c_0\}.$ Moreover, $\{a,c\} \subset N_G(b)$ and $a, c \not \in A \cup C.$ Finally, since $|N_G(b)| = k=2(\mu-1)+2$, we must have that
    \begin{gather*}
        N_G(b) = \{a,c\} \sqcup (A \setminus \{a_0\}) \sqcup (C  \setminus \{c_0\}).
    \end{gather*}
    For vertices $x, y \in G$, we define $\mathbf{1}_{x \sim y}$ to be $1$, if $x \sim y$, and $0$, otherwise. Also, let $\deg_W(x)=|N_G(x) \cap W|.$ We have that
    \begin{align*}
        \mu &= |N_G(b) \cap N_G(d)| =  |\{a,c\} \cap N_G(d)| + |(A \setminus \{a_0\}) \cap N_G(d)|) + | (C  \setminus \{c_0\}) \cap N_G(d)|\\
        & = 2 + |(A \setminus \{a_0\}) \cap N_G(d)| + | (C  \setminus \{c_0\}) \cap N_G(d)|\\
        &= 2 + |A \cap N_G(d)| + | C \cap N_G(d)| - \mathbf{1}_{d \sim a_0} - \mathbf{1}_{d \sim c_0}\\
        &= 2 + \Big( (\mu-1) - \deg_W(d)\Big) + \Big( (\mu-1) - \deg_W(d)\Big)- \mathbf{1}_{d \sim a_0} - \mathbf{1}_{d \sim c_0}\\
    \end{align*}
    In particular, 
    \begin{gather*}
        2\deg_W(d) = \mu - \mathbf{1}_{d \sim a_0} - \mathbf{1}_{d \sim c_0} \implies \deg_W(d) \leq \frac{\mu}{2},
    \end{gather*}
    and so,
    \begin{gather*}
        |N_G(d) \cap A| = (\mu-1) - \deg_W(d) \geq \frac{\mu}{2}-1.
    \end{gather*}
    Similarly, we can show that
    \begin{gather*}
        |N_G(d) \cap C| \geq \frac{\mu}{2}-1.
    \end{gather*}
If $\mu \geq 5$, then $|N_G(d) \cap A| \geq \frac{\mu}{2}-1$ implies that $|N_G(d)\cap A|>1$ and consequently, $|N_G(d)\cap (A\setminus \{a_0\})|\geq 1$. Because $b$ is adjacent to every vertex in $A\setminus \{a_0\}$, we deduce that there is $z \in A\setminus \{a_0\}$ such that  $z \sim a$, $z\sim b$, and $z \sim d$. 

This implies that, in $G_{b,d},$ the vertex $a$ does not form a single vertex connected component. Similarly, there is $z' \in C\setminus \{c_0\}$ that is a common neighbor of $b,c,$ and $d$. This shows that the vertex $c$ also does not form a single vertex connected component in $G_{b,d}$. By Proposition \ref{prop:common-neighborhood-two-components}, both $a,c$ must be in the same component of $G_{b,d}$. This finishes the proof.
\end{proof}

A slightly weaker version of the following lemma is implicit in the proof of Lemma 5.19 in \cite{cioaba-guo-ji-mim}. Here, we prove a slightly stronger version that is valid over any field.  
\begin{lemma} \label{lemma:path-in-common-neighborhood}
    Let $G$ be a graph, $\FF$ a field, $C=(a,b,c,d)$ a cycle in $G$, and $v_1=b,\dots,v_s=d$ a path between $b$ and $d$ such that $\{v_1,\dots,v_s\} \subset N_G(a) \cap N_G(c)$. For $i\in \{1,\ldots,s-1\}$, let $C_i=\{a,v_i,v_{i+1}\}$ and $C_i'=\{c,v_i,v_{i+1}\}$. 
    
    There are $c_1,\dots,c_{s-1},c_1',\dots,c_{s-1}' \in \{-1,1\} \subset \FF$ such that 
    \begin{gather*}
        T(C)=\sum_{i=1}^{s-1} c_i\delta_1^T(C_i)+c_i'\delta_1^T(C_i').
    \end{gather*}
\end{lemma}

\begin{proof}
    In the following, $\delta_1$ will denote $\delta_{1,\FF}$ and $T$ will denote $T_{\FF}$. We use induction on $s$. For $i\in \{1,\ldots,s-1\}$, let $e_i=\{v_i,v_{i+1}\}$. If $s=2$, then $b \sim d$ and $e_1=\{b,d\}.$ Denote:
    \begin{gather*}
        X_1 := [\{a,b\}, \delta_1^T(C_1)]\cdot \delta_1^T(C_1) - [\{a,b\}, \delta_1^T(C_1)] \cdot [e_1, \delta_1^T(C_1)]\cdot [e_1, \delta_1^T(C_1')] \cdot \delta_1^T(C_1').  
    \end{gather*}
    Observe that in the $\im \delta_1^T \subset \ker \delta_0^T$ element $X_1$, the edge $\{a,b\}$ appears with coefficient $1$, and the term $e_1$ vanishes. Since the support of $X_1$ is the set of edges in the cycle $(a,b,c,d),$ $X_1=c'T(a,b,c,d)$ for some $c'.$ By comparing the coefficients of the edge $\{a,b\}$ in $X_1$ and $T(a,b,c,d),$ we conclude that $c=1$ and $X_1=T(a,b,c,d).$ Thus, the proposition is true when $s=2.$ 
    
    Assume the proposition holds for all $s' < s.$ There are constants $c_1,\dots,c_{s-2},c_1',\dots,c'_{s-2} \in \{-1,1\}$ such that:
    \begin{gather*}
        T(a,b,c,v_{s-1}) = \sum_{i=1}^{s-2} c_i\delta_1^T(C_i)+c_i'\delta_1^T(C_i').
    \end{gather*}
    Let $C'=(a,b,c,v_{s-1})$ and define:
    \begin{gather*}
        c_{s-1} = -[\{a,v_{s-1}\}, T(C')] \cdot [\{a,v_{s-1}\}, \delta_1^T(C_{s-1})], \text{ and}\\
        c'_{s-1} = -[\{c,v_{s-1}\}, T(C')] \cdot [\{c,v_{s-1}\}, \delta_1^T(C'_{s-1})].
    \end{gather*}
    Clearly, both $c_{s-1}, c'_{s-1} \in \{1,-1\}$. Consider now:
    \begin{gather*}
        Y = T(C') + c_{s-1} \delta_1^T(C_{s-1}) + c'_{s-1} \delta_1^T(C'_{s-1}).
    \end{gather*}
    By construction, the edges $\{a,v_{s-1}\}$ and $\{c,v_{s-1}\}$ vanish in $Y$. The support of $Y$ consists of the edges of the cycle $C$, along with possibly the edge $e_{s-1}.$ However, by construction, $Y \in \im \delta_1^T \subset \ker \delta_0^T,$ and consequently, $Y$ cannot have an edge with an endpoint not on any other edge in the support of $Y$. The edge $e_{s-1}$ must also vanish in $Y$. Hence, the support of $Y$ is precisely the set of edges in the cycle $C=(a,b,c,d),$ and $Y=c''T(a,b,c,d)$ for some constant $c''.$ Finally, we observe that the coefficient of the edge $\{a,v_1\}$ is $1$ in both $Y$ and $T(C).$ Therefore, we must have $Y=T(C).$ This completes the proof of the lemma.
\end{proof}

\begin{prop} \label{prop:conference-4-cycle}
    Let $G$ be a conference graph on $v \geq C_0'$ vertices, and $\FF$ a field. If $C=(a,b,c,d)$ is a cycle in $G$, then $T_{\FF}(C) \in \im \delta_{1,\FF}^T.$
\end{prop}

\begin{proof}
    In the following, $\delta_1$ will denote $\delta_{1,\FF}^T$ and $T$ will denote $T_{\FF}$. If $a \sim c$ or $b \sim d,$ the proposition follows by Lemma \ref{lemma:cycle-split-over-chord} or the base case of Proposition \ref{lemma:path-in-common-neighborhood}. We now assume that $C$ is an induced cycle. By Theorem \ref{thm:conference-induced-4-cycle-across-vertices}, either $a,c$ are in the same connected component of $G_{b,d},$ or, $b,d$ are in the same connected component of $G_{a,c}.$ Without loss of generality, we may assume that $b,d$ are in the same connected component of $G_{a,c}.$ Thus, there are vertices $x_1=b,\dots,x_s=d$ in $N_G(a) \cap N_G(c)$ for some $s$ such that $x_i \sim x_{i+1},i=1,\dots,s-1.$ Denoting $C_i=\{a,x_i,x_{i+1}\}, C_i'=\{c,x_i,x_{i+1}\},i=1,\dots,s-1,$ by Lemma \ref{lemma:path-in-common-neighborhood}, there are constants $c_i,c_i' \in \{1,-1\} \subset \FF, i=1,\dots,s-1$ such that
    \begin{gather*}
         T(C)=\sum_{i=1}^{s-1} c_i\delta_1^T(C_i)+c_i'\delta_1^T(C_i').
    \end{gather*}
    This shows that $T(C) \in \im \delta_1^T.$ This completes the proof of the proposition. 
\end{proof}

We now present and prove the main theorem of the section.
\begin{theorem} \label{thm:vanishing-homology-conference-graphs}
    If $G$ is a conference graph on $v \geq C_0'$ vertices and $\FF$ is a field, then 
    \begin{equation*}
        \im \delta_{1,\FF}^T(G) = \ker \delta_{0,\FF}^T(G) \text{ and } \hlcl{G}{\FF}=0.
    \end{equation*}
\end{theorem}

\begin{proof}
    The proof follows by combining Proposition \ref{prop:conference-4-cycle} with the arguments in the proof of Theorem 5.21 in \cite{cioaba-guo-ji-mim}. In the following, $\delta_1$ and $T$ will denote $\delta_{1,G,\FF}$ and $T_{\FF},$ respectively. All coefficients and equations below are over $\FF$, so we will drop the subscript $\FF$ for the rest of this proof. 
    
    We will prove that if $C=(x_1,\dots,x_{\ell})$ is an arbitrary cycle in $G$, then $T(C) \in \im \delta_1^T$. 
    
    We use strong induction on $\ell.$ We prove the cases $\ell \in \{3,4,5\}$ first. 
    
    If $\ell=3,$ the claim is immediate because $T(C)=\delta_1^T(C)$ or $T(C)=-\delta_1^T(C).$ If $\ell=4$, then by Proposition \ref{prop:conference-4-cycle}, $T(C)\in \im \delta_1^T$. 
    
    Now, consider $\ell=5.$ If $C$ is not an induced cycle, then Lemma \ref{lemma:cycle-split-over-chord} and the results in the previous paragraph imply that $T(C) \in \im \delta_1^T$. 
    
    We may now assume that $C$ is induced. We consider two cases. 
    
    Case 1: $N_G(x_1) \cap N_G(x_3) \cap N_G(x_4) \ne \emptyset.$
    
    Let $x_6$ be a common neighbor of $x_1,x_3,$ and $x_4.$ Denote the cycles:
    \begin{gather*}
        C_1 = (x_1,x_2,x_3,x_6), \ C_2 = (x_6,x_3,x_4), \ C_3=(x_6,x_4,x_5,x_1).
    \end{gather*}
    Also, denote
    \begin{gather*}
        c_2 = -[\{x_6,x_3\}, T(C_1)] \cdot [\{x_6,x_3\}, T(C_2)].
    \end{gather*}
    The term $\{x_6,x_3\}$ vanishes in the $\ker \delta_0^T$ element $X_1:=T(C_1)+c_2T(C_2).$ Hence, the support of $X_1$ consists of the edges of the cycle $C_4:=(x_1,x_2,x_3,x_4,x_6).$ However, the coefficient of the edge $\{x_1,x_2\}$ is $1$ in both $X_1$ and $T(C_4).$ Thus, we must have $X_1=T(C_4).$ Also, define:
    \begin{gather*}
        c_3 = -[\{x_6,x_4\}, T(C_4)] \cdot [\{x_6,x_4\}, T(C_3)].
    \end{gather*}
    In the $\ker \delta_0^T$ element $X_2:=T(C_4)+c_3T(C_3)$, the term $\{x_4,x_6\}$ vanishes, and thus, the support of $X_2$ consists of the edges of the cycle $C$, along with possibly the edge $\{x_1,x_6\}.$ However, the graph created by the support edges of a $\ker \delta_0^T$ element cannot have pendant edges, so, we must have that $X_2$ has support precisely the set of edges in the cycle $C.$ Finally, the edge $\{x_1,x_2\}$ appears with coefficient $1$ in both $X_2$ and in $T(C),$ so, we conclude that $T(C)=X_2=T(C_1)+c_2T(C_2)+c_3T(C_3).$ Finally, the cycles $C_1,C_2,C_3$ have lengths at most $4,$ and so, by the above discussions, $T(C_i)\in \im \delta_1^T, i=1,2,3.$ Consequently, $T(C) \in \im \delta_1^T.$ \par 
    
    Case 2: $N_G(x_1) \cap N_G(x_3) \cap N_G(x_4) = \emptyset.$ 
    
    In this case, $\big(N_G(x_1) \cap N_G(x_3)\big) \cap \big(N_G(x_1) \cap N_G(x_4)\big) = \emptyset$. Because $|N_G(x_1) \cap N_G(x_3)|=|N_G(x_1) \cap N_G(x_4)|=\mu=\frac{k}{2}=\frac{|N_G(x_1)|}{2},$ $N_G(x_1)\cap N_G(x_3)$ and $N_G(x_1) \cap N_G(x_4)$ partition $N_G(x_1).$ Because $v>9$ and \cite[Prop. 5.2]{cioaba-guo-ji-mim}, $N_G(x_1)$ is connected. Thus, there are $x_6 \in N_G(x_1) \cap N_G(x_3)$ and $x_7 \in N_G(x_1) \cap N_G(x_4)$ such that $x_6 \sim x_7$. We cannot have both $x_6=x_2$ and $x_7=x_5$ since $C$ is induced. Thus, at most one of $x_6=x_2$ and $x_7=x_5$ can hold. 
    
    We first consider the sub-case $x_6 \ne x_2$ and $x_7 \ne x_5.$  Define the cycles:
    \begin{gather*}
        C_1 = (x_1,x_2,x_3,x_6), \ C_2=(x_6,x_3,x_4,x_7), \ C_3 = (x_1,x_7,x_6),\ C_4=(x_7,x_4,x_5,x_1). 
    \end{gather*}
    Also, let
    \begin{gather*}
        c_2 = -[\{x_6,x_3\}, T(C_1)] \cdot [\{x_6,x_3\}, T(C_2)], \text{ and,}\\
        X_1 = T(C_1)+c_2T(C_2).
    \end{gather*}
    The term $\{x_6,x_3\}$ vanishes in $X_1.$ Hence, the support of $X_1$ is the set of edges in the cycle $C_5=(x_1,x_2,x_3,x_4,x_7,x_6)$. Because the edge $\{x_1,x_2\}$ appears with coefficient $1$ in both, $X_1$ and $T(C_5)$, we conclude that $X_1=T(C_5)$. Define now:
    \begin{gather*}
        c_3= -[\{x_6,x_7\}, T(C_5)] \cdot [\{x_6,x_7\}, T(C_3)], \text{ and,}\\
        X_2 = T(C_5)+c_3T(C_3).
    \end{gather*}
    In the $\ker \delta_0^T$ element $X_2,$ the term $\{x_6, x_7\}$ vanishes, and so, the support of $X_2$ consists of the set of edges in the cycle $C_6:=(x_1,x_2,x_3,x_4,x_7),$ along with possibly the edge $\{x_1,x_6\}.$ However, since $X_2 \in \ker \delta_0^T,$ the graph created by the support edges of $X_2$ cannot have a pendant edge. Thus, the support of $X_2$ must be the set of edges in $C_6.$ By comparing the coefficient of the edge $\{x_1,x_2\}$ in $X_2$ and $T(C_6)$ (which is $1$ in both cases), we conclude that $X_2 = T(C_6).$ Finally, we define:
    \begin{gather*}
        c_4 = -[\{x_4,x_7\}, T(C_6)] \cdot [\{x_4,x_7\}, T(C_4)], \text{ and,}\\
        X_3 = T(C_6) + c_4 T(C_4).
    \end{gather*}
    By a similar argument as above, we observe that the edge $\{x_4,x_7\}$ vanishes in the $\ker \delta_0^T$ element $X_3,$ and hence the support of $X_3$ must be the set of edges of the cycle $C.$ By observing that the coefficient of the edge $\{x_1,x_2\}$ is $1$ in both $X_3$ and $T(C),$ we conclude that $X_3=T(C).$ That is, 
    \begin{gather*}
        T(C) = T(C_1) + c_2T(C_2) + c_3 T(C_3) + c_4 T(C_4).
    \end{gather*}
    By the induction hypothesis, $T(C_i) \in \im \delta_1^T$ for $i\in \{1,2,3,4\},$ and consequently, $T(C) \in \im \delta_1^T$, as required.\par

    On the other hand, if one of $x_6=x_2$ and $x_7=x_5$ is true, we may assume without loss of generality that $x_7=x_5$ but $x_6 \neq x_2.$ Define the cycles:
    \begin{gather*}
        C_1 = (x_1,x_2,x_3,x_6), \quad C_2=(x_1,x_6,x_5), \quad C_3=(x_6,x_3,x_4,x_5). 
    \end{gather*}
    and coefficients $c_2,c_3 \in \{1,-1\}$ by:
    \begin{gather*}
        c_2 = -[\{x_1,x_6\}, T(C_1)] \cdot [\{x_1,x_6\}, T(C_2)],\quad c_3 = -[\{x_3,x_6\}, T(C_1)] \cdot [\{x_3,x_6\}, T(C_3)].
    \end{gather*}
    Also, let
    \begin{gather*}
        X = T(C_1) + c_2T(C_2) + c_3T(C_3) \in \ker(\delta_0^T).
    \end{gather*}
    By construction, the terms $\{x_1,x_6\}$ and $\{x_6,x_3\}$ vanish in $X$ and the support of $X$ consists of the edges of the cycle $C,$ along with possibly the edge $\{x_5,x_6\}.$ However, since $X \in \ker(\delta_0^T),$ the pendant edge $\{x_5,x_6\}$ cannot be in its support. The support of $X$ is the edge-set of the cycle $C$. Furthermore, the coefficient of the edge $\{x_1,x_2\}$ is $1$ in both $T(C)$ and in $X$ by construction. Thus, we conclude that
    \begin{gather*}
        T(C) = X =T(C_1) + c_2T(C_2) + c_3T(C_3).
    \end{gather*}
     By the induction hypothesis, $T(C_i) \in \im(\delta_1^T)$ for $ i\in \{1,2,3\}$. Thus, $T(C) \in \im(\delta_1^T)$ as well. This completes the proof for the case $\ell=5.$
    
    We now consider the case $\ell \geq 6.$ 
    
    If $C$ is not an induced cycle, then, without loss of generality, we assume that $x_1 \sim x_i$ for some $3 \leq i \leq \ell-1.$ By a similar argument as before, we can show that $T(C)=T(C_1)+c_2T(C_2),$ where
    \begin{gather*}
        C_1 = (x_1,x_2,\dots,x_i), \ C_2 = (x_i,x_{i+1},\dots,x_{\ell},x_1).
    \end{gather*}
    Here, both the cycles $C_1$ and $C_2$ have length at most $\ell-1,$ and thus, by the induction hypothesis, $T(C_i)\in \im \delta_1^T$ for $i\in \{1,2\},$ and consequently, $T(C) \in \im \delta_1^T.$ 
    
    If $C$ is an induced cycle, then, $x_1 \not \sim x_4$. There is some $x \in V(G), x \not = x_i, 1 \leq i \leq \ell,$ such that $x_1 \sim x \sim x_4.$ Denote the cycles
    \begin{gather*}
        C_1 = (x_1,x_2,x_3,x_4,x), \ C_2 = (x_1,x,x_4,x_5,\dots,x_{\ell}).
    \end{gather*}
    Using an argument similar to above, we can show that 
    \begin{gather*}
        T(C) = T(C_1)-[\{x_1,x\}, T(C_1)] \cdot [\{x_1,x\}, T(C_2)] \cdot T(C_2). 
    \end{gather*}
Here, both cycles $C_1$ and $C_2$ have lengths strictly less than $\ell$, and, by the induction hypothesis, $T(C_1), T(C_2) \in \im \delta_1^T$.
    Hence, $T(C) \in \im \delta_1^T$. This proves the claim every $\ell,$ and the proof of the theorem follows.
\end{proof}

\section{Strongly Regular Graphs With Smallest Eigenvalue $-2$}
\label{sec:strongly-regular-graphs-with-smallest-positive-eigenvalue--2}

In this section, in Theorem \ref{thm:nonvanishing-classification--2}, we prove that the only strongly regular graphs with smallest adjacency eigenvalue $-2$ and a non-vanishing $H_1$ over some field are the cycle $C_4$, the Petersen graph, the Shrikhande graph and the lattice graphs.

We recall Seidel's classification of strongly regular graphs with smallest adjacency eigenvalue $-2$.

\begin{theorem}[Theorem 9.2.1, \cite{BH}] \label{thm:seidel--2-classification}
If $G$ is a strongly regular graph with smallest eigenvalue $-2$, then $G$ is one of the following:
    \begin{enumerate}
        \item the complete $n$-partite graph $K_{n \times 2}$ with all partite sets of size two and parameters\\ $(v,k,\lambda,\mu)=(2n,2n-2,2n-4,2n-2), \ n\geq 2,$
        \item the lattice graph $L_2(n)=K_n \square K_n,$ with parameters $(v,k,\lambda,\mu)=(n^2,2(n-1),n-2,2), \ n \geq 3,$
        \item the triangular graph $T(n)$ with parameters $(v,k,\lambda,\mu)=(\binom{n}{2},2(n-2),n-2,4), \ n \geq 5,$
        \item the Shrikhande graph, with parameters $(v,k,\lambda,\mu)=(16,6,2,2),$
        \item one of the three Chang graphs, with parameters $(v,k,\lambda,\mu)=(28,12,6,4),$
        \item the Petersen graph, with parameters $(v,k,\lambda,\mu)=(10,3,0,1),$
        \item the Clebsch graph, with parameters $(v,k,\lambda,\mu)=(16,10,6,6),$
        \item the Schl{\"a}fli graph, with parameters $(v,k,\lambda,\mu)=(27,16,10,8).$
    \end{enumerate}
\end{theorem}

We first deal with the complete multipartite graphs $K_{n\times 2}$.

\begin{prop} \label{prop:homology-k_nx2}
    If $\FF$ is a field, then
    \begin{enumerate}
        \item $\hlcl{K_{n \times 2}}{\FF} \cong \FF$ if $n=2,$ and,
        \item $\hlcl{K_{n \times 2}}{\FF} = 0$ for $n \geq 3.$
    \end{enumerate}
\end{prop}
\begin{proof}
    When $n=2$, $K_{n \times 2}$ is the cycle of length four. In this case, 
    $\dkdt{0,\FF}{K_{n \times 2}}=1, \didt{1,\FF}{K_{n \times 2}}=0,$ and hence $\hlcl{K_{n \times 2}}{\FF}\cong \FF.$ 
    
    If $n \geq 3$, then denote $G=K_{n \times 2}$. We may describe $G$ as:
    \begin{gather*}
        V(G) = \{a_i,b_i : 1\leq i\leq n\}, \quad E(G)=\bigcup_{1 \leq i<j \leq n} \{\{a_i,a_j\}, \{b_i,b_j\}, \{a_i,b_j\}, \{a_j, b_i\}\},
    \end{gather*}
    for some distinct vertices $a_1,\ldots,a_n,b_1,\ldots,b_n$.
    
    An induced path of length three in $G$ must have one of the following forms for some $i \ne j$:
    \begin{gather*}
        (a_i,a_j,b_i), \quad (a_i,b_j,b_i), \quad (b_i,b_j,a_i), \quad (b_i,a_j,a_i).
    \end{gather*}
    Consequently, an induced cycle of length four is of the form:
    \begin{gather*}
        (a_i,a_j,b_i,b_j), \quad i \ne j,
    \end{gather*}
    and $G$ has no induced cycle of length $5$ or more. For $i\ne j$, let $C=(a_i,a_j,b_i,b_j)$ be an induced $4$-cycle in $G$. Because $n\geq 3$, there is $k \in \{1,\dots,n\} \setminus \{i,j\},$ and both $a_k$ and $b_k$ are common neighbors of the vertices $a_i,a_j,b_i,b_j.$ Hence, every induced cycle in $G$ of length at least $4$ has $4$ consecutive vertices with a common neighbor in $G$. Thus, $G$ satisfies the hypothesis of Theorem \ref{thm:4-vertex-common-neighbor}, and we conclude that $\hlcl{G}{\FF}=0.$ 
\end{proof}

The following result was proved for $\FF=\re$ in \cite{cioaba-guo-ji-mim}.
\begin{prop} \label{prop:homology-T(n)}
Let $n \geq 5$ and $\FF$ a field. If $T(n)$ is the triangular graph on $\binom{n}{2}$ vertices, then $\hlcl{T(n)}{\FF}=0$.
\end{prop}
\begin{proof}
    Denote $G=T(n)$. Any induced cycle $C$ of length four in $G$ must have the form:
    \begin{gather*}
        C=(\{i,j\}, \{j,k\}, \{k,l\}, \{i,l\})
    \end{gather*}
    for some pairwise distinct symbols $i,j,k,l \in \{1,\dots,n\}.$ For such a cycle $C$, the vertex $\{i,k\}$ of $G$ is a common neighbor of the vertices of $C.$ Hence, the vertices of any induced cycle of length four have a common neighbor in $G$. Furthermore, any induced path $P$ of length three in $G$ has the following form for some pairwise distinct symbols $i,j,k,l,m:$
    \begin{gather*}
        P=(\{i,j\},\{j,k\},\{k,l\},\{l,m\}). 
    \end{gather*}
    The vertex $\{j,l\}$ of $G$ is a common neighbor of the four vertices on $P.$ This implies that any four consecutive vertices on an induced cycle of length five or more have a common neighbor in $G$. Thus, $G$ satisfies the hypothesis of Theorem \ref{thm:4-vertex-common-neighbor}, and we conclude that $\hlcl{G}{\FF}=0.$ 
\end{proof}

When $\FF=\re$, the following result is implicit in \cite[Theorem 4.1]{cioaba-guo-ji-mim}.

\begin{prop} \label{prop:homology-lattice-graphs}
    If $G=K_n \square K_n, n \geq 2$, is the lattice graph on $n^2$ vertices and $\FF$ is a field, then
    \begin{gather*}
        \dim(\hlcl{G}{\FF})=n(n-2)+1>0.
    \end{gather*}
\end{prop}
\begin{proof}
The graph $G$ has strongly regular graph parameters $(v,k,\lambda,\mu)=(n^2,2(n-1),n-2,2).$ Therefore, $\dim(\ker(\delta_{0,G,\FF}^T))=\frac{2n^2(n-1)}{2}-n^2+1=n^2(n-2)+1.$ 

Let $\T$ be the set of all triangles in $G,$. For $i\in \{1,\ldots,n\}$, let $\T_i$ and $\T_i'$ be the set of triangles in $G$ whose vertices are in the row $i$ and column $i$, respectively. 
Clearly, $\T_1,\ldots,\T_n,\T'_1,\ldots, \T'_n$ form a partition of $\T$. Thus, $\delta_{1,G,\FF}^T=\bigoplus_{i=1}^{2n} \delta_{1,K_n,\FF}^T.$ Using the fact that $\rnkg{\FF}{\delta_{1,K_n,\FF}^T}=\binom{n-1}{2}$ (see \cite[Thm. 8]{meshulam-newman-rabinovich} for example), we get that 
    \begin{gather*}
        \rnkg{\FF}{\delta_{1,G,\FF}^T}=(2n) \rnkg{\FF}{\delta_{1,K_n,\FF}^T}=2n\binom{n-1}{2}=n(n-1)(n-2).
    \end{gather*}
    Therefore,
    \begin{gather*}
        \dim(\hlcl{G}{\FF})=\dim(\ker(\delta_{0,G,\FF}^T))-\dim(\im(\delta_{1,G,\FF}^T)) \\=n^2(n-2)+1-n(n-1)(n-2)=n(n-2)+1,
    \end{gather*}
    as needed.
\end{proof}

If $M$ is an integer matrix and $\FF$ a field, we will denote by $\rankg{\FF}{M}$ the rank of the matrix $M$ viewed as a matrix over $\FF$. We now recall some basic facts regarding the Smith normal  Form of integer matrices. 

For a nonzero $m \times n$ integer matrix $M$, there is an integer $m\times m$ matrix $U$, an integer $m\times n$ diagonal matrix $D$, and an integer $n\times n$ matrix $V$ such that $M=UDV$ and $\det(U), \det(V) \in \{-1,1\}$. Also, the non-zero diagonal entries $d_1,\dots,d_r$ of $D$ satisfy $d_i | d_{i+1}$ for $i\in \{1,\ldots,r-1\}$. The integers $d_1,\dots,d_r$ are called the invariant factors of $M$, and are independent of $U$ and $V$. 

The following result is well-known, but, for the sake of completeness, we provide a short proof. See also \cite{BH} and \cite{stanley-smithnormalformcombinatorics} for other results on the Smith Normal Form.

\begin{lemma} \label{lemma:invariant-factor-and-f-rank}
    Let $M$ be a nonzero integer matrix with nonzero invariant factors $d_1, \dots, d_r.$ If $\FF$ is a field, then
    \begin{enumerate}
        \item if $\fchar{\FF}=0,$ $\rankg{\FF}{M}=r,$
        \item if $\fchar{\FF}=p$ for some prime $p$, 
        \begin{gather*}
            \rankg{\FF}{M}= | \{1 \leq i \leq r : p \text{ does not divide } d_i\}|.
        \end{gather*}
    \end{enumerate}
    In particular, if $d_i\in \{-1,1\}$ for $i\in \{1, \ldots, r\}$, then 
    \begin{gather*}
            \rankg{\FF}{M}=r.
    \end{gather*}
\end{lemma}

\begin{proof}
    We first observe that if $\FF_1 \subset \FF_2$ are arbitrary fields and $N$ a matrix over $\FF_1,$ then,
    \begin{gather*}
        \rnkg{\FF_1}{N} = \text{number of nonzero rows in a row-echelon form of $N$ as a matrix over $\FF_1$} \\
        = \text{number of nonzero rows in a row-echelon form of $N$ as a matrix over $\FF_2$}=\rnkg{\FF_2}{N}. 
    \end{gather*}
    Let $M$ be of size $m \times n,$ $D$ the $m \times n$ diagonal matrix with diagonal $(d_1,\dots,d_r,0,\dots,0),$ and $U,V$ square integer matrices of orders $m$ and $n$ respectively with $\det(U), \det(V) \in \{1,-1\}$ and $M=UDV.$ Then, $U^{-1}$ and $V^{-1}$ are both integer matrices, and thus, $U,V$ are invertible over $\FF.$ We then have:
    \begin{gather*}
        \rankg{\FF}{M}=\rankg{\FF}{UDV}=\rankg{\FF}{D}.
    \end{gather*}
    Now, if $\fchar{\FF}=0,$ then $d_1,\dots,d_r$ are nonzero elements of $\FF,$ and, $\rankg{\FF}{M}=\rankg{\FF}{D}=r.$ On the other hand, if $\fchar{\FF}=p$ for some prime $p$, then the prime subfield of $\FF$ is the field on $p$ elements, $\FF_p$. By the observation above, we have:
    \begin{gather*}
        \rnkg{\FF}{D}=\rnkg{\FF_p}{D}=| \{1 \leq i \leq r : p \text{ does not divide } d_i\}|.
    \end{gather*}
    This proves the first two statements of the lemma. The third statement follows, since if $d_i=1,i=1,\dots,r,$ then $| \{1 \leq i \leq r : p \text{ does not divide } d_i\}|=r,$ and in either case, we have $\rankg{\FF}{M}=\rankg{\FF}{D}=r.$
\end{proof}

\begin{corollary}
Let $G$ be a graph with $n$ vertices, $m$ edges, and $c$ connected components. If $\delta_{1,G,\Z}^T$ has exactly $m-n+c$ invariant factors, each of them being $1$ or $-1$, then $\hlcl{G}{\FF}=0$, for any field $\FF$.
\end{corollary}
\begin{proof}
   Let $\FF$ be a field. The hypothesis and Lemma \ref{lemma:invariant-factor-and-f-rank} imply that 
    $$\rankg{\FF}{\delta_{1,G,\FF}^T}=\rankg{\FF}{\delta_{1,G,\Z}^T}=m-n+c=\dim(\ker(\delta_{0,G,\FF}^T)).$$ 
    Hence, $\hlcl{G}{\FF}=0$.
\end{proof}

\begin{prop} \label{prop:computational-data}
Let $G$ be a graph and $\FF$ be a field.

\begin{enumerate}
        \item If $G$ is the Shrikhande graph with parameters $(16,6,2,2)$, then\\
        $\dim(\hlcl{G}{\FF})=2$.
        
        \item If $G$ is any of the following graphs:
        \begin{enumerate}
            \item the Clebsch graph, with parameters $(v,k,\lambda,\mu)=(16,10,6,6),$
            \item the Schl{\"a}fli graph, with parameters $(v,k,\lambda,\mu)=(27,16,10,8)$,
            \item one of the three Chang graphs, with parameters $(v,k,\lambda,\mu)=(28,12,6,4),$
        \end{enumerate}
        then $\hlcl{G}{\FF}=0$, for every field $\FF:$
    \end{enumerate}
\end{prop}
\begin{proof}

    Let $\FF$ be a field. 

    \begin{enumerate}
\item If $G$ is the Shrikhande graph, then its cycle space $\ker(\delta_{0,\FF}^T)$ has dimension $\dim(\ker(\delta_{0,\FF}^T))=|E(G)|-|V(G)|+1=33$. 
    
    By computer, we verified (see \cite{srg-sporadic-cases-code}) that $\delta_{1,G,\Z}^T$ has exactly $31$ nonzero invariant factors, each invariant factor being $1$. Hence, $\rnkg{\FF}{\delta_{1,G,\FF}^T}=31$ and $\dim(\hlcl{G}{\FF})=33-31=2$.

 \item   Now, let $G$ be one of the graphs listed in the statement of the Proposition. If the parameters of $G$ are $(v,k,\lambda,\mu)$, then the cycle space of $G$ has dimension $\dim(\ker(\delta_{0,G,\FF}^T))=\frac{vk}{2}-v+1$. 
 
 By computer, we verified that (see \cite{srg-sporadic-cases-code}) $\delta_{1,G,\Z}^T$ has exactly $\frac{vk}{2}-v+1$ nonzero invariant factors, each invariant factor being $1.$ Thus, $\rnkg{\FF}{\delta_{1,G,\FF}^T}=\frac{vk}{2}-v+1,$ and consequently,
    \begin{gather*}
        \dim(\hlcl{G}{\FF})=\dim(\ker(\delta_{0,G,\FF}^T))-\rnkg{\FF}{\delta_{1,G,\FF}^T}=\left(\frac{vk}{2}-v+1\right)-\left(\frac{vk}{2}-v+1\right)=0.
    \end{gather*}
\end{enumerate}
\end{proof}

\begin{theorem} \label{thm:nonvanishing-classification--2}
Let $\FF$ be a field. If $G$ is a strongly regular graph with smallest adjacency eigenvalue $-2$ such that $\hlcl{G}{\FF}\not = 0$, then $G$ must be one of the following graphs:
    \begin{enumerate}
        \item the cycle graph on $4$ vertices with parameters $(v,k,\lambda,\mu)=(4,2,0,2)$,
         \item the Petersen graph, with parameters $(v,k,\lambda,\mu)=(10,3,0,1).$
         \item the Shrikhande graph, with parameters $(v,k,\lambda,\mu)=(16,6,2,2),$       
        \item the lattice graph $L_2(n)=K_n \square K_n,$ with parameters $(v,k,\lambda,\mu)=(n^2,2(n-1),n-2,2), \ n \geq 3,$
        
        % \item the triangular graph $T(5)$ with parameters $(v,k,\lambda,\mu)=(10,6,3,4),$
        % \item one of the three Chang graphs, with parameters $(v,k,\lambda,\mu)=(28,12,6,4),$
       
        % \item the Clebsch graph, with parameters $(v,k,\lambda,\mu)=(16,10,6,6),$
        % \item the Schl{\"a}fli graph, with parameters $(v,k,\lambda,\mu)=(27,16,10,8).$
    \end{enumerate}
\end{theorem}

\begin{proof}
The result follows from Theorem \ref{thm:seidel--2-classification}, Proposition \ref{prop:homology-k_nx2}, Proposition \ref{prop:homology-T(n)}, and Proposition \ref{prop:computational-data}.
\end{proof}

\section{A Classification Theorem} \label{sec:a-classification-theorem}

 We recall the following theorems, see \cite[Thm. 8.6.4]{BvM}, \cite{KLYC}, or \cite{neumaier}.
\begin{theorem}[\cite{neumaier}, Theorem 5.1] \label{thm:neumaier-1}
    Let $m \geq 2$ be an integer. If $G$ is a strongly regular graph with $\lambda_{\min}(G)=-m$, then $G$ belongs to at least one of the following sub-families:
    \begin{enumerate}
        \item complete multi-partite graphs with one or more parts, each of size $m$,
        \item strongly regular graphs associated with orthogonal arrays $\oa{m}{n}$ for some $n$ (note that Neumaier calls this class Latin square graphs and uses the notation $LS_m(n)$),
        \item block graph of Steiner systems $S(2,m,n)$ for some $n$, (note that Neumaier denotes this class by $S_m(n)$),
        \item a finite family $E_m,$ which we will call the exceptional family.
    \end{enumerate}
\end{theorem}

\begin{theorem}[\cite{neumaier}, Theorem 5.2] \label{thm:neumaier-2}
    Any strongly regular graph that does not belong to one of the above sub-families for some $m$ belongs to at least one of the following families:
    \begin{enumerate}
        \item conference graphs,
        \item disjoint union of cliques of the same size.
    \end{enumerate}
\end{theorem}

We will also need the following proposition.

\begin{prop} \label{prop:homology-complete-multipartite}
    Let $G$ be a complete multipartite graph with $r \geq 1$ parts, each of size $m \geq 3$. Let $\FF$ be a field. Then,
    \begin{enumerate}
        \item if $r=1,$ then $G$ is an edgeless graph on $m$ vertices, and $\hlcl{G}{\FF}=0,$
        \item if $r=2,$ then $G$ is a complete bipartite graph $K_{m,m}$, and $\dim(\hlcl{G}{\FF})=(m-1)^2.$ In particular, $\hlcl{G}{\FF} \not = 0,$
        \item if $r \geq 3,$ then $\hlcl{G}{\FF}=0.$
    \end{enumerate}
\end{prop}
\begin{proof}
The case $r=1$ is immediate because $\didt{1,\FF}{G}=\dkdt{0,\FF}{G}=0.$ 

If $r=2$, then Theorem \ref{thm:dim-cycle-space} implies that $\dkdt{0,\FF}{G}=m^2-2m+1=(m-1)^2 > 0$, while $\didt{1,\FF}{G}=0$ since $G$ has no triangles. Hence,
    \begin{gather*}
        \dim(\hlcl{G}{\FF}) = \dkdt{0,\FF}{G}-\didt{1,\FF}{G}=(m-1)^2.
    \end{gather*}

If $r \geq 3$, then the proof is similar to the one of Proposition \ref{prop:homology-k_nx2}. Any induced $4$-cycle of $G$ has the form $(a,b,c,d),$ where $a,c$ belong to the same partite set in $G$ while $c,d$ belong to a different partite set of $G$. Since $r \geq 3,$ we may choose a vertex from any other part that is a common neighbor of $a,b,c,d$. Also, $G$ contains no induced cycle of length $5$ or more. Thus, the hypothesis of Theorem \ref{thm:4-vertex-common-neighbor} holds for $G$, and we conclude that $\hlcl{G}{\FF}=0.$
\end{proof}

We now present the main result of this section.

\begin{theorem} \label{thm:classification-theorem-1}
    Let $G$ be a strongly regular graph and $\FF$ a field. If $\hlcl{G}{\FF} \not = 0$, then $G$ must be one of the following graphs:
    \begin{enumerate}
       \item the Petersen graph, with parameters $(v,k,\lambda,\mu)=(10,3,0,1),$
       \item the Shrikhande graph, with parameters $(v,k,\lambda,\mu)=(16,6,2,2),$
       \item an element of the exceptional family $E_m$, for some $m \geq 3$, described in the Neumaier's classification,
       \item a conference graph $G(v, \frac{v-1}{2}, \frac{v-5}{4}, \frac{v-1}{4})$ for some $v \leq 255$,
       \item a complete bipartite graph with each part of size $n$ and parameters $(2n, n, 0, n),$
       \item a lattice graph $L_2(n)=K_n \square K_n,$ with parameters $(v,k,\lambda,\mu)=(n^2,2(n-1),n-2,2), \ n \geq 3$.
    \end{enumerate}
\end{theorem}

\begin{proof}
If $G$ has integral adjacency eigenvalues and its smallest adjacency eigenvalue is $-2,$ then use Theorem \ref{thm:nonvanishing-classification--2}.
     
If $G$ has integral eigenvalues and its smallest eigenvalue is $-3$ or less, then by Theorem \ref{thm:neumaier-1}, Theorem \ref{theorem:vanishing-h1-orthogonal-array}, Theorem \ref{thm:vanishing-homology-conference-graphs} and Proposition \ref{prop:homology-complete-multipartite}, $G$ belongs to $E_m$ for some $m \geq 3$ or $G$ is a complete bipartite graph or $G$ is a strongly regular graph associated with $OA(3,2), \ OA(3,3),$ and $OA(4,3)$ (which are $K_4, K_{3,3,3},$ and $K_9$, respectively). By Theorem \ref{thm:4-vertex-common-neighbor-srg}, $\hlcl{K_4}{\FF}=0$ and $\hlcl{K_9}{\FF}=0$ for every field $\FF.$ By Proposition \ref{prop:homology-complete-multipartite}, $\hlcl{K_{3,3,3}}{\FF}=0$ for every field $\FF$. Thus, $G$ cannot be $K_4, K_{3,3,3},$ or $K_9$, and $G$ must be a complete bipartite graph in this case or a member of $E_m$ for some $m \geq 3.$

% If $G$ is a complete multipartite graph with each part having some fixed size $m \geq 3,$ then use Theorem \ref{thm:classification-theorem-1}, Theorem \ref{thm:classification-theorem-2},

% By Proposition \ref{prop:homology-complete-multipartite}, $G$ must be a complete bipartite graph. Assume that $G$ is not a complete multipartite graph. By Theorems \ref{theorem:vanishing-h1-orthogonal-array} and \ref{thm:vanishing-h1-steiner-system}, $G$ cannot be the block graph of an Steiner System or the graph associated with an orthogonal array $OA(m,n)$ with $n \geq 4.$ Hence, $G$ must be the graph associated with an $OA(m,n)$ with $m \geq 3, n \leq 4$, or, $G$ must belong to the exceptional family $E_m.$ In any orthogonal array $OA(m,n)$, $m \leq n+1$ holds, thus, the only possible orthogonal arrays satisfying $m \geq 3, n \leq 3$ are $OA(3,2), \ OA(3,3),$ and $OA(4,3).$ Now, the strongly regular graphs corresponding to $OA(3,2), \ OA(3,3),$ and $OA(4,3)$ are respectively $K_4, K_{3,3,3},$ and $K_9$, respectively. By Theorem \ref{thm:4-vertex-common-neighbor-srg}, $\hlcl{K_4}{\FF}=0$ and $\hlcl{K_9}{\FF}=0$ for every field $\FF.$ Also, by Proposition \ref{prop:homology-complete-multipartite}, $\hlcl{K_{3,3,3}}{\FF}=0$ for every field $\FF$. Thus, $G$ cannot be $K_4, K_{3,3,3},$ or $K_9$.
    
% Finally, if $G$ does not have integral adjacency eigenvalues or  with smallest adjacency eigenvalue $-3$ or less, then, by Theorem \ref{thm:neumaier-2}, 
If $G$ does not have integral eigenvalues or if $G$ has integral eigenvalues and the smallest adjacency eigenvalue is at least $-2$, then $G$ must be a disjoint union of cliques or conference graph. If $G$ is a disjoint union of cliques of size at least $3,$ then $\hlcl{G}{\FF}=0.$ On the other hand, if $G$ is a disjoint union of edges, we also have $\hlcl{G}{\FF}=0$. Finally, if $G$ is a conference graph, by Theorem \ref{thm:vanishing-homology-conference-graphs}, $G$ must have fewer than $256$ vertices, i.e., at most $255$ vertices. This completes the proof of the theorem.
\end{proof}

As an immediate corollary, we obtain the following theorem.

\begin{theorem} \label{thm:classification-theorem-2}
    If $(G_n)_{n=1}^{\infty}$ is a sequence of pairwise distinct strongly regular graphs and $(\FF_n)_{n=1}^{\infty}$ is a sequence of arbitrary fields such that $\hlcl{G_n}{\FF_n}\not =0$ for every $n$, then one of the following statements is true:
    \begin{enumerate}
        \item for infinitely many values of $n$, $G_n$ is a lattice graph with at least nine vertices,  
        \item $\lim_{n \to \infty}\lambda_{\min}(G_n)=-\infty.$
    \end{enumerate}
\end{theorem}

\section{Conclusion}

In this paper, we study the first homology groups of strongly regular and determined 
several sufficient conditions for these groups to be non-trivial. It would be interesting to study which graphs in the exceptional families $E_m$ for $m \geq 3$ in Neumaier's classification of strongly regular graphs have nontrivial first homology group.

In the course of our investigation, we encountered the question of determining, for a strongly regular graph $G$ and non-adjacent vertices $x$ and $y$, if  $G_{x,y}$, the subgraph induced by the common neighbors of $x$ and $y$, is connected or not. We studied this problem for Paley graphs and conference graphs in general, and we hope to extend our results in the future. 

In Theorem \ref{thm:conference-induced-4-cycle-across-vertices}, we showed that if $G$ is a conference graph with at least $C_0'$ vertices, then for any induced four-cycle $(a,b,c,d)$ in $G$, either $a,c$ are in the same connected component of $G_{b,d},$ or $b,d$ are in the same connected component of $G_{a,c}.$ We showed that $C_0'=256$ suffices. It would be interesting to study whether $C_0'=256$ is optimal.

It would be interesting to extend our study to other families of graphs such as distance-regular graphs, graphs in association schemes or Cayley graphs. 

\newpage
\bibliographystyle{plain}
\bibliography{references}
\end{document}